\documentclass[12pt,letterpaper,reqno]{amsart}
\usepackage{amsfonts}
\usepackage{amsmath}
\usepackage{amsthm, amscd}
\usepackage{mathtools}
\usepackage{amssymb}
\usepackage{mathrsfs}
\usepackage{graphicx}
\usepackage{caption}
\usepackage{subcaption}
\usepackage{float}
\usepackage{xypic}
\usepackage[hidelinks]{hyperref}
\usepackage[abs]{overpic}
\usepackage[alphabetic]{amsrefs}
\usepackage{etex}
\usepackage{tikz-cd}

\theoremstyle{plain}\newtheorem{Theorem}{Theorem}[section]
\theoremstyle{plain}\newtheorem{Corollary}[Theorem]{Corollary}
\theoremstyle{plain}\newtheorem{Lemma}[Theorem]{Lemma}
\theoremstyle{plain}\newtheorem{Definition}[Theorem]{Definition}
\theoremstyle{plain}\newtheorem{Proposition}[Theorem]{Proposition}
\theoremstyle{plain}\newtheorem{Question}[Theorem]{Question}
\theoremstyle{plain}
\theoremstyle{plain}
\theoremstyle{plain}\newtheorem*{Claim*}{Claim}
\theoremstyle{plain}\newtheorem*{Theorem*}{Theorem}

\makeatletter
\newtheorem*{rep@theorem}{\rep@title}
\newcommand{\newreptheorem}[2]{%
\newenvironment{rep#1}[1]{%
 \def\rep@title{#2 \ref{##1}}%
 \begin{rep@theorem}}%
 {\end{rep@theorem}}}
\makeatother
\newreptheorem{Theorem}{Theorem}
\newreptheorem{Proposition}{Proposition}

\theoremstyle{remark}\newtheorem{remark}[Theorem]{Remark}
\theoremstyle{remark}
\theoremstyle{remark}\newtheorem*{Notation*}{Notation}

\DeclareMathOperator{\cl}{cl}
\DeclareMathOperator{\Emb}{Emb}
\DeclareMathOperator{\id}{id}
\DeclareMathOperator{\supp}{supp}
\DeclareMathOperator{\dist}{dist}
\DeclareMathOperator{\pt}{pt}
\DeclareMathOperator{\Ima}{Im}

\author{Dan Cristofaro-Gardiner}
\address{Department of Mathematics, The University of Maryland at College Park, Maryland 20742, USA}
\email{dcristof@umd.edu}
\author{Boyu Zhang}
\address{Department of Mathematics, The University of Maryland at College Park, Maryland 20742, USA}
\email{bzh@umd.edu}

\title[Topological symplectic and bi-Lipschitz structures]{Topological symplectic manifolds and bi-Lipschitz structures}

\begin{document}
\date{}
\maketitle
\begin{abstract}
We show that a topological symplectic manifold has a canonically associated bi-Lipschitz structure.  As a corollary, we obtain the first examples of non-existence and non-uniqueness for topological symplectic structures.  Our arguments hold for any topological manifold admitting an atlas with transition maps that are $C^0$--limits of bi-Lipschitz homeomorphisms.
\end{abstract}

\section{Introduction}

Eliashberg--Gromov have shown that symplectic maps possess the following remarkable rigidity property: a $C^0$--limit of symplectic diffeomorphisms which is a diffeomorphism is automatically symplectic.  Building on this, one can define a {\em symplectic homeomorphism} to be a homeomorphism which is locally the $C^0$--limit of symplectic diffeomorphisms, and
a {\em topological symplectic manifold} to be a topological manifold with an atlas whose charts have transition maps that are symplectic homeomorphisms; see e.g. the exposition in \cite[Ch.2]{vincent}.  

Though the notion of a topological symplectic manifold is a fundamental one that has attracted interest for decades\footnote{See for example the discussion in \cite[Sec. 2.7]{vincent}.}, not much is known about it. 
In fact, even the topological consequences of the definition have not been much studied.  This is the goal of the present work.  We define a {\em topological bi-Lipschitz manifold} to be a manifold with an atlas whose charts have transition functions that are locally $C^0$--limits of bi-Lipschitz homeomorphisms.  Then, if a manifold is topologically symplectic, it must be topologically bi-Lipschitz. Recall that a {\em bi-Lipschitz} structure on a topological manifold is an atlas whose transition maps are bi-Lipschitz homeomorphisms.
The main theorem of our paper is that every topological bi-Lipschitz $4$--manifold admits a canonical bi-Lipschitz structure, which implies a rather remarkable amount of geometric structures.

\begin{Theorem}
\label{thm_main_bilip}
Every topological bi-Lipschitz four-manifold has a canonically associated bi-Lipschitz structure. 
\end{Theorem}

We defer the precise statement of what is meant by ``canonical" in this context to Theorem \ref{thm_uniqueness_of_C_structure}.

Bi-Lipschitz manifolds have a number of desirable properties.  For example, they have a notion of differential forms, and a de Rham complex computing the cohomology of the manifold.  Our main interest here, though, is in the implications for the existence and uniqueness of topological symplectic structures.  We start with the following, which is the first known obstruction to the existence of topological symplectic structures on even-dimensional manifolds:

\begin{Corollary}
There exist closed topological $4$--manifolds that do not admit any topological symplectic structure. 
\end{Corollary}

The corollary follows, because it was shown in \cite{sullivan1989quasiconformal} that Donaldson theory extends to bi-Lipschitz and quasi-conformal $4$--manifolds, and so can be used to obstruct the existence of such structures on topological $4$-manifolds.
More specifically, by \cite[Section 7(ii)]{sullivan1989quasiconformal}, every simply connected topological $4$--manifold with an even, negative definite, intersection form that represents $(-2)$ cannot admit bi-Lipschitz or quasi-conformal structures. Hence by Theorem \ref{thm_main_bilip}, such manifolds cannot admit topological symplectic structures. 

We have a similar non-uniqueness statement.

\begin{Corollary}
There exist two different topological symplectic manifolds that are homeomorphic but not symplectically homeomorphic.
\end{Corollary}

Here, a homeomorphism between two topological symplectic manifolds is called a \emph{symplectic homeomorphism} if it takes every topological symplectic chart to a topological symplectic chart (see also \cite[Section 2.2]{vincent}). 
The corollary again follows via the work of Donaldson--Sullivan.  By Theorem \ref{thm_main_bilip} (more precisely, Theorem \ref{thm_uniqueness_of_C_structure}), every symplectic homeomorphism induces a bi-Lipschitz equivalence.  Donaldson--Sullivan \cite[Section 7(iii)]{sullivan1989quasiconformal} showed that $\mathbb{C}P^2 \# 8 \overline{\mathbb{C}P^2}$ and the ``Barlow surface" are not bi-Lipschitz equivalent, though they are homeomorphic; both of these support symplectic structures.

We note that our methods also give the following result.  
We say that a manifold is a \emph{topological quasi-conformal manifold} if it admits a chart whose transition functions are $C^0$--limits of quasi-conformal homeomorphisms.

\begin{Theorem}
\label{thm_main_qc}
Every topological quasi-conformal four-manifold admits a canonical quasi-conformal structure. 
\end{Theorem}

Our proofs extend Sullivan's arguments in \cite{sullivan1979hyperbolic}. Sullivan generalized Kirby's torus trick to the bi-Lipschitz and quasi-conformal categories by replacing the torus with closed, almost parallelizable hyperbolic manifolds, and proved that every topological manifold of dimension $\neq 4$ admits a bi-Lipschitz structure. The dimension restriction is needed because homeomorphisms of $\mathbb{R}^n$ ($n \neq 4$) can be $C^0$--approximated by PL homeomorphisms, and this property fails in dimension $4$. Nevertheless, if a $4$--manifold admits a chart with transition functions that are $C^0$--limits of bi-Lipschitz or quasi-conformal homeomorphisms, we will show that Sullivan's argument extends to such manifolds.

The main difficulty is that being a $C^0$--limit of bi-Lipschitz (resp. quasi-conformal) homeomorphisms is not clearly a local condition. More precisely, if $f \colon U \to \mathbb{R}^4$ is an open embedding and there is an open cover $\{U_\alpha\}$ of $U$ such that $f$ is a $C^0$--limit of bi-Lipschitz (resp.\ quasi-conformal) embeddings on each $U_\alpha$, it does not follow by definition that $f$ is a $C^0$--limit of bi-Lipschitz (resp.\ quasi-conformal) embeddings of $U$. This creates an essential difficulty because Sullivan's smoothing argument makes use of a sequence of \emph{local} modifications for every open embedding $f:U\to \mathbb{R}^n$, and a straightforward extension of the argument would require that $f$ be a $C^0$--limit of bi-Lipschitz (resp. quasi-conformal) embeddings of $U$ after each step. We resolve this in Section \ref{sec_local_Cbar_to_global_Cbar}, where we prove in Theorem \ref{thm_local_Cbar_to_global_Cbar} that the existence of local approximations implies global approximations on compact subsets. 
This is facilitated by
introducing a class of isotopies, which will be called \emph{universal $\overline{\mathcal{C}}$ isotopies} (Definition \ref{def_universal_Cbar_isotopy}), that preserve  global $\overline{\mathcal{C}}$ properties on every subdomain. We show that the universal $\overline{\mathcal{C}}$ condition can be verified locally and is invariant under coordinate changes, and that the smoothing argument can be executed in a way where every step is given by a universal $\overline{\mathcal{C}}$ isotopy. 

The paper is organized as follows.
Section \ref{sec_preliminary} introduces the basic notation and establishes several elementary lemmas. Section \ref{sec_first_consideration} shows that, by a standard argument, the main results can be deduced from a smoothing theorem for embeddings of open subsets in $\mathbb{R}^4$. Section \ref{sec_local_Cbar_to_global_Cbar} addresses the main technical difficulty: we show that if an open embedding can be locally $C^0$ approximated by bi-Lipschitz (resp. quasi-conformal) embeddings, then it can be $C^0$ approximated by bi-Lipschitz (resp. quasi-conformal) embeddings on every compact subset. We then use this to verify the desired smoothing result, which finishes the proof of the main theorems. In Section \ref{sec_appendix}, we provide the details for the proof of an ``extension after restriction'' result (Theorem \ref{thm_match_handle_in_C}) that was used in Section \ref{sec_local_Cbar_to_global_Cbar}; while the relevant arguments are known, we found it difficult to locate direct references for the specific formulations we need, so we include the proofs for the convenience of the reader and to make the paper more self-contained.

\begin{remark}
Other definitions of symplectic homeomorphism have been considered, and these give alternate definitions of a topological symplectic manifold;  it is not at present clear whether or not these are equivalent.   The most common alternate definition is that one sometimes defines a symplectic homeomorphism to be a (global) $C^0$ limit of symplectic diffeomorphisms; this has the advantage of tending to preserve symplectic capacities.  This leads to an equivalent definition to topological symplectic manifolds, and hence our arguments work just as well with this definition.  There are also {\em capacity preserving homeomorphisms} \cite{vincent}, i.e. homeomorphisms which preserve a normalized symplectic capacity; our arguments do not directly imply that a manifold with transition maps that are capacity preserving homeomorphisms is bi-Lipschitz.  For more about these notations, see e.g. \cite{vincent}.
\end{remark}

\subsection{Future directions}

Recall that a bi-Lipschitz homeomorphism between open subsets of $\mathbb{R}^{2n}$ is called {\em symplectic} if it pulls back the standard symplectic form to itself almost everywhere.  A {\em symplectic bi-Lipschitz} structure is an atlas with transition maps that are symplectic bi-Lipschitz homeomorphisms. It seems crucial to understand the following potential refinement of Theorem~\ref{thm_main_bilip}.

\begin{Question}
\label{que:sympbilip}
    Can every topological symplectic structure be canonically refined to a symplectic bi-Lipschitz one?
\end{Question}

If this question has an affirmative answer, it would imply, for example, via \cite{dusan} that spheres other than $S^2$ never support a topological symplectic structure, resolving a longstanding question of Hofer discussed in e.g. \cite[Ch. 2.7]{vincent}, \cite[Problem 41]{mcduffsalamon}.  However, it is of more general interest.  For example, it was speculated to us by D. Sullivan that a symplectic bi-Lipschitz structure could allow for the theory of pseudoholomorphic curves to extend.  Inspired by this, we ask the following question, which we have purposely left somewhat open-ended:

\begin{Question}
Does the theory of pseudoholomorphic curves extend to topological symplectic manifolds?
\end{Question}

One hopes in other words to find a canonical refinement to a symplectic bi-Lipschitz structure via Question~\ref{que:sympbilip} and then explore Sullivan's suggestion.

\subsection*{Acknowledgments.}
A huge debt is owed to Dennis Sullivan 
for suggesting that results like the ones we establish should hold
and for proposing the general strategy.
We are also indebted to Dennis for sharing inspiring speculations,
and for many helpful conversations and correspondence.  
We thank Vincent Humili\`ere and Seyfaddini Sobhan for their interest in this work and helpful comments.

DCG thanks the NSF for their support under agreements DMS-2227372 and DMS-2238091.
BZ thanks the NSF for their support under agreement DMS-2405271 and the
Simons Foundation for a travel grant.

\section{Preliminaries}
\label{sec_preliminary}
Suppose $n$ is an integer and $n\ge 2$. Later, we will focus on the case that $n=4$.
Let $B^n$ denote the closed unit ball in $\mathbb{R}^n$. For $a>0$, let $B^n(a)$ denote the closed ball in $\mathbb{R}^n$ centered at the origin with radius $a$. If $n=0$, our convention is $B^n(a) = B^n =\{0\}$ for all $a>0$. If $a<b$, let $B^n((a,b))$ denote $\{x\in\mathbb{R}^n\mid a<|x|<b\}$, let $B^n([a,b])$ denote $\{x\in\mathbb{R}^n\mid a\le |x| \le b\}$, and define $B^n((a,b])$, $B^n([a,b))$ analogously. 

If $A\subset \mathbb{R}^n$, let $\cl(A)$ denote the closure of $A$, let $\mathring{A}$ denote the interior of $A$. 
If $A$ is a topological manifold with boundary, we use $\partial A$ to denote its boundary manifold. By definition, $\partial B^0=\emptyset$. 

If $V\subset U$ are open subsets of $\mathbb{R}^n$, we write $V\Subset U$ if the closure of $V$ is a compact subset of $U$.

Let $I=[0,1]$.
For $A\subset \mathbb{R}^n$, the \emph{support} of an embedding $f:A\to \mathbb{R}^n$ is the closure of $\{x\in A\mid f(x) \neq x\}$ in $\mathbb{R}^n$. The \emph{support} of a homotopy $F:A\times I\to \mathbb{R}^n$ is the closure of the set 
\[
\{x\in A\mid \exists\, t_1<t_2 \,\,\textrm{such that}\,\, f(x,t_1) \neq f(x,t_2)\}.
\]
A homotopy $F:A\times I\to \mathbb{R}^n$ is called an \emph{isotopy} if $F(-,t)$ is an embedding of $A$ for each $t\in [0,1]$. 
For $\epsilon>0$, a homotopy (resp. isotopy) $F$ is called an \emph{$\epsilon$--homotopy} (resp. \emph{$\epsilon$--isotopy}), if $\sup_{x,t_1,t_2} |F(x,t_1)-F(x,t_2)| < \epsilon$. If $U\subset \mathbb{R}^n$ is an open subset, we say that an embedding, homotopy, or isotopy is \emph{compactly supported in $U$}, if the support is a compact subset of $U$. 

An open embedding $f:U\to \mathbb{R}$ is called \emph{bi-Lipschitz}, if there exists $K\ge 1$ such that 
\[
\frac{1}{K}|x-y|\le |f(x)-f(y)| \le K|x-y|
\]
for all $x,y\in U$. It is called \emph{quasi-conformal}, if there exists $K\ge 1$ such that 
\[
\limsup_{r\to 0} \frac{\max_{|y-x|=r} |f(x) - f(y)|}{\min_{|y-x|=r} |f(x) - f(y)|}\le K
\]
for all $x\in U$.

An open embedding $f:U\to \mathbb{R}^n$ is called \emph{locally bi-Lipschitz} (resp. \emph{locally quasi-conformal}), if there is an open cover $\{U_\alpha\}$ of $U$ such that $f$ is bi-Lipschitz (resp. quasi-conformal) on each $U_\alpha$. Note that the corresponding constant $K$ may depend on $\alpha$. 

Lemmas \ref{lem_local_C_implies_global_C}, \ref{lem_inverse_C} below are well-known properties of bi-Lipschitz and quasi-conformal embeddings.

\begin{Lemma}
\label{lem_local_C_implies_global_C}
Suppose $V\Subset U$ are open subsets of $\mathbb{R}^n$. 
\begin{enumerate}
    \item If an embedding $f:U\to \mathbb{R}^n$ is locally quasi-conformal, then $f$ is quasi-conformal on $V$.
    \item If an embedding $f:U\to \mathbb{R}^n$ is locally bi-Lipschitz, then $f$ is bi-Lipschitz on $V$.
\end{enumerate}
\end{Lemma}

\begin{proof}
    Part (1) follows directly from the definition. For Part (2), the Lipschitz estimate for $f^{-1}|_{f(V)}$ is given by \cite[Lemma 2.17]{tukia1981lipschitz}. To show that $f|_V$ is Lipschitz, by the compactness of $cl(V)\subset U$ and the assumption that $f$ is locally Lipschitz, there exists $\epsilon>0$ and $K_1>0$ such that $|f(x)-f(y)|\le K_1|x-y|$ for all $x,y\in V$ with $|x-y|<\epsilon$. Let $M>0$ be an upper bound of $|f|$ on $V$, let $K = \max \{K_1, 2M/\epsilon\}$, then $f$ is $K$--Lipschitz on $V$. 
\end{proof}

In the following, we use \emph{$\mathcal{C}$ embeddings} to denote either locally quasiconformal embeddings or locally bi-Lipschitz embeddings. If $f:U\to \mathbb{R}^n$ is an open embedding, we use $f^{-1}$ to denote the composition $f(U)\stackrel{f^{-1}}{\to} U\hookrightarrow \mathbb{R}^n$. 

\begin{Lemma}
\label{lem_inverse_C}
If $f$ is an open $\mathcal{C}$ embedding, then $f^{-1}$ is an open $\mathcal{C}$ embedding. 
\end{Lemma}

\begin{proof}
    Recall that we assume $n\ge 2$. If $\mathcal{C}$ denotes locally quasi-conformal embeddings, the result follows from the fact that quasi-conformal maps on open subsets of $\mathbb{R}^n$  are locally quasi-symmetric, see \cite[Theorem 11.14]{heinonen2001lectures}.  If $\mathcal{C}$ denotes locally bi-Lipschitz embeddings, the result follows from the definition.
\end{proof}

As a consequence of Lemma \ref{lem_inverse_C}, the set of $\mathcal{C}$ embeddings forms a pseudo group.
Recall that by definition, a \emph{pseudo group} is a set $\mathcal{S}$ of embeddings of open subsets of $\mathbb{R}^n$ in $\mathbb{R}^n$, such that: 
(1) If $\{U_\alpha\}$ is an open cover of $U$, then $f:U\to \mathbb{R}^n$ is in $\mathcal{S}$ if and only if $f|_{U_\alpha}\in\mathcal{S}$ for all $\alpha$, (2) $f\in\mathcal{S}$ if and only if $f^{-1}\in\mathcal{S}$, (3) If $f:U\to \mathbb{R}^n$ and $g:f(U)\to \mathbb{R}^n$ are both in $\mathcal{S}$, then $g\circ f\in \mathcal{S}$.  In general, every pseudo group of embeddings defines a category of manifolds.
In particular, one defines a \emph{$\mathcal{C}$ manifold} to be a topological manifold together with a system of charts whose transition functions are all $\mathcal{C}$ embeddings. The notions of \emph{$\mathcal{C}$ maps} and \emph{$\mathcal{C}$ homeomorphisms} between $\mathcal{C}$ manifolds are defined similarly.

We prove the following technical lemma on the inverses of open embeddings for later reference. 

\begin{Lemma}
\label{lem_inverse_top_embedding}
    Suppose $\{f_i\}_{i\in \mathbb{N}}$, $f$ are open embeddings from $U$ to $\mathbb{R}^n$, and $\{f_i\}$ converges to $f$ in $C^0$. Suppose $V\Subset f(U)$. Then 
    \begin{enumerate}
        \item $V\subset f_i(U)$ for $i$ sufficiently large.
        \item $\{f_i^{-1}|_{V}\}$ converges to $f^{-1}|_V$ in $C^0$.
    \end{enumerate}
\end{Lemma}

\begin{proof}
    Let $B_x(r)$ denote the closed Euclidean ball centered at $x$ with radius $r$. For each $x\in f^{-1}(\cl(V))$, choose $r_x>0$ such that $B_x(r_x)\subset U$. Let $d_x$ be the distance between $f(x)$ and $f(\partial B_x(r_x))$. If $i$ is sufficiently large so that $\|f_i-f\|_{C^0}<d_x/2$, then $B_{f(x)}(d_x/2)$ is contained in the bounded component of $\mathbb{R}^n\setminus f_i(\partial B_x(r_x))$, so $B_{f(x)}(d_x/2)\subset f_i(B_x(r_x))$, and hence $B_{f(x)}(d_x/2)\subset f_i(U)$. Since $V\Subset f(U)$, one can take finitely many $x\in f^{-1}(\cl(V))$ such that the balls $B_{f(x)}(d_x/2)$ cover $V$, so Part (1) of the lemma is proved.

    For each $\epsilon>0$ such that $\epsilon<\dist(f^{-1}(V), \mathbb{R}^n\setminus U)$, let 
    \[
    d(\epsilon) = \inf_{x\in f^{-1}(V)} \dist(f(x), f(\partial B_x(\epsilon))).
    \]
    By the compactness of $\cl(f^{-1}(V))$, we have $d(\epsilon)>0$.  By the previous argument, if $\|f_i-f\|_{C^0}<d(\epsilon)/2$ and $x\in f^{-1}(V)$, then $B_{f(x)}(d(\epsilon)/2)\subset f_i(B_x(\epsilon))$. Hence $f(x)\in f_i(B_x(\epsilon))$, so $|f_i^{-1}(f(x))-f^{-1}(f(x))|=|f_i^{-1}(f(x))-x|\le \epsilon$.  Therefore, $\|f^{-1}|_V-f_i^{-1}|_V\|_{C^0}\le \epsilon$. This proves Part (2).
\end{proof}

Let $\overline{\mathcal{C}}$ be the set of open embeddings that are locally given by $C^0$--limits of $\mathcal{C}$ embeddings.

\begin{Lemma}
\label{lem_inverse_Cbar}
    If $f\in \overline{\mathcal{C}}$, then $f^{-1}\in \overline{\mathcal{C}}$. 
\end{Lemma}

\begin{proof}
    Let $U$ denote the domain of $f$, and suppose $\{U_\alpha\},\{U_\alpha'\}$ are open covers of $U$ such that $U_\alpha'\Subset U_\alpha$ and each $f|_{U_\alpha}$ is a $C^0$--limit of $\mathcal{C}$ embeddings of $U_\alpha$ in $\mathbb{R}^n$. Suppose $f|_{U_\alpha}$ is the $C^0$--limit of $g_\alpha^i: U_\alpha\to \mathbb{R}^n$. By Lemma \ref{lem_inverse_top_embedding} (1), after removing finitely many terms, we may assume that the range of each $g_\alpha^i$ contains $f(U_\alpha')$. Let $h_\alpha^i=(g_\alpha^i)^{-1}|_{f(U_\alpha')}$. By Lemma \ref{lem_inverse_top_embedding} (2), $h_\alpha^i$ converge to $f^{-1}|_{f(U_\alpha')}$ in $C^0$. So $f^{-1}|_{f(U_\alpha')}$ is a $C^0$--limit of $\mathcal{C}$ embeddings.
\end{proof}

As a consequence, $\overline{\mathcal{C}}$ embeddings form a pseudo group. We define the notions of $\overline{\mathcal{C}}$ manifolds, $\overline{\mathcal{C}}$ maps, $\overline{\mathcal{C}}$ homeomorphisms accordingly.

\begin{remark}
\label{rmk_n=4}
    Although all results in the rest of the paper work for general $n$, we will focus on the case for $n=4$. This is because when $n\neq 4$, every open embedding $f: U \to \mathbb{R}^n$ is in $\overline{\mathcal{C}}$. In fact,  every collared embedding of a closed ball extends to a homeomorphism of $\mathbb{R}^n$ by the Schoenflies theorem, and when $n\neq 4$, every homeomorphism of $\mathbb{R}^n$ can be $C^0$--approximated by PL homeomorphisms (\cite{connell1963approximating, bing1963stable, kirby1969stable} for $n\ge 5$, and \cite{moise1952affine4,bing1959alternative} for $n\le 3$). In this case, the analogous results to Theorems \ref{thm_main_bilip}, \ref{thm_main_qc} were already proved in \cite{sullivan1979hyperbolic}.
\end{remark}

Unlike the case for $\mathcal{C}$ embeddings, the analogous result to Lemma \ref{lem_local_C_implies_global_C} for $\overline{\mathcal{C}}$ embeddings is not obvious from the definition, because local approximations of homeomorphisms may not glue to global approximations. Therefore, we introduce the following notation.

\begin{Definition}
We say that an open embedding $f:U\to\mathbb{R}^n$ is a \emph{globally $\overline{\mathcal{C}}$ embedding}, if it can be written as the $C^0$--limit of a sequence of $\mathcal{C}$ embeddings of $U$.
\end{Definition}

We will show in Section \ref{sec_local_Cbar_to_global_Cbar} that if $f:U\to \mathbb{R}^n$ is a $\overline{\mathcal{C}}$ embedding and $V\Subset U$, then $f|_V$ is a globally $\overline{\mathcal{C}}$ embedding.

If $A\subset \mathbb{R}^n$ is not necessarily open, we say that a topological embedding $f:A\to \mathbb{R}^n$ is a \emph{$\mathcal{C}$ embedding} (resp.  \emph{$\overline{\mathcal{C}}$ embedding}) if it can be extended to a $\mathcal{C}$ (resp. $\overline{\mathcal{C}}$) embedding of an open neighborhood of $A$. We say that $f:A\to \mathbb{R}^n$ is a \emph{globally $\overline{\mathcal{C}}$ embedding}, if it is the $C^0$--limit of a sequence of $\mathcal{C}$ embeddings of $A$. 

For $A, B\subset \mathbb{R}^n$, let 
\[
\Emb(A;B),\quad \Emb_{\mathcal{C}}(A;B),\quad \Emb_{\overline{\mathcal{C}}}(A;B)
\]
denote the sets of topological, $\mathcal{C}$, and $\overline{\mathcal{C}}$ embeddings respectively of $A$ in $\mathbb{R}^n$ such that the image is contained in $B$. If $A\subset B$, we use $\id$ to denote the identity embedding of $A$. If $A'\subset A\cap B$, let $\Emb(A,A';B)\subset \Emb(A;B)$ denote the subset consisting of embeddings that equal the identity on $A'$, and define $\Emb_{\mathcal{C}}(A,A';B)$, $\Emb_{\overline{\mathcal{C}}}(A,A';B)$ similarly. 
Unless otherwise specified, all embedding spaces are endowed with the $C^0$--topology, and all maps between embedding spaces are assumed continuous.

\begin{Definition}
If $A\subset B$, $X$ are sets, and $f:A\to X$ and $g:B\to X$ are set-theoretic maps, the \emph{extension of $f$ to $B$ by $g$} is defined to be the map $\hat f:B\to X$ such that $\hat f(x) = f(x)$ if $x\in A$, and $\hat f(x) = g(x)$ if $x\in B\setminus A$. 
\end{Definition}

For later reference, we record the following standard lemmas about extensions of embeddings. 
\begin{Lemma}
\label{lem_extend_cpt_supported_homeo}
    Suppose $U\subset \mathbb{R}^n$ is open and $f:U\to \mathbb{R}^n$ is a compactly supported embedding. Let $\hat f$ be the extension of $f$ to $\mathbb{R}^n$ by the identity. Then 
    \begin{enumerate}
        \item $\hat f$ is a homeomorphism of $\mathbb{R}^n$.
        \item $f(U)=U$.
        \item If $f$ is a $\mathcal{C}$ embedding, then $\hat f$ is a $\mathcal{C}$ embedding.
    \end{enumerate}
\end{Lemma}
\begin{proof}
    It is clear that $\hat f$ is a local homeomorphism and is locally $\mathcal{C}$ if $f$ is, and that Statement (2) is a consequence from Statement (1). We only need to show that $\hat f$ is a homeomorphism. Note that for each compact set $A\subset \mathbb{R}^n$, we have $\hat{f}^{-1}(A) \subset A\cup \supp f$, which is compact. Since $\hat f$ is a local homeomorphism, this implies that $\hat f$ is a covering map. Since $\mathbb{R}^n$ is simply connected, we conclude that $\hat f$ is a homeomorphism.
\end{proof}

\begin{Lemma}
\label{lem_extend_emb_cpt_supported_difference}
    Suppose $V'\Subset V\Subset U$, $f:U\to \mathbb{R}^n$ is an embedding, and $g:V\to \mathbb{R}^n$ is an embedding such that $f|_{V\setminus V'} = g|_{V\setminus V'}$. Let $\hat g$ be the extension of $g$ to $U$ by $f$. Then 
    \begin{enumerate}
        \item $\hat g$ is an embedding.
        \item $\hat g(V) = f(V)$, $\hat g(U) = f(U)$.
        \item $\hat g$ is a $\mathcal{C}$ embedding if both $f$ and $g$ are $\mathcal{C}$ embeddings. 
    \end{enumerate}

\end{Lemma}
\begin{proof}
    By the assumptions, $g\circ f^{-1}:f(V)\to \mathbb{R}^n$ is a compactly supported embedding. By Lemma \ref{lem_extend_cpt_supported_homeo}, the extension of $g\circ f^{-1}|_{f(V)}$ to $\mathbb{R}^n$ by the identity is a homeomorphism $h:\mathbb{R}^n\to \mathbb{R}^n$. Note that $\hat{g} = h\circ f$, so $\hat{g}$ is an embedding. By Lemma \ref{lem_extend_cpt_supported_homeo}(2), we have $h(f(V)) = f(V)$, so $\hat{g}(V) = f(V)$, and hence $\hat{g}(U) = f(U)$. Statement (3) is also clear. 
\end{proof}

\section{First considerations}
\label{sec_first_consideration}
This section begins the proof of our main theorems and reviews some arguments from the literature that we will need.

\subsection{$\mathcal{C}$-smoothings of $\overline{\mathcal{C}}$-embeddings }

The first key point is that we can reduce the proof of all of our theorems to the following smoothing statement:

\begin{Proposition}
\label{prop_prop_C_smoothing_Cbar_embeddings}
Suppose $U$ is an open subset of $\mathbb{R}^n$, and $A,A',B,B'$ are open subsets of $U$ such that $A' \Subset A$, $B'\Subset B$. Suppose $f:U\to \mathbb{R}^n$ is a $\overline{\mathcal{C}}$ embedding such that $f|_{B}$ is a $\mathcal{C}$ embedding.  Then for every $\epsilon>0$, there exists an $\epsilon$--isotopy $F:U\times I \to \mathbb{R}^n$, such that the following statements hold.
    \begin{enumerate}
        \item $F(-,0)=f$
        \item $F(-,1)$ is a $\mathcal{C}$ embedding on $A'\cup B'$.
        \item $F$ is compactly supported in $A$, and $F(-,t)$ is a $\overline{\mathcal{C}}$ embedding for each $t$.
    \end{enumerate}
\end{Proposition}

Establishing Proposition \ref{prop_prop_C_smoothing_Cbar_embeddings} is the main difficulty that needs to be overcome to prove the theorems.
We defer its proof to Section \ref{sec_local_Cbar_to_global_Cbar}. In this section, we explain why Proposition \ref{prop_prop_C_smoothing_Cbar_embeddings} implies our main theorems.  This is a standard argument (cf., for example, \cite[Section 4]{munkres2016elementary}):

\begin{Theorem}
\label{thm_construct_C_chart}
    Every $\overline{\mathcal{C}}$ manifold admits a $\overline{\mathcal{C}}$ chart whose transition functions are in $\mathcal{C}$.
\end{Theorem}

\begin{proof}[Proof, assuming Proposition \ref{prop_prop_C_smoothing_Cbar_embeddings}]
Suppose $M$ is a $\overline{\mathcal{C}}$ manifold. Then there exist countable open covers $\{U_a\}_{a\in \mathbb{Z}^+}$, $\{U_a'\}_{a\in \mathbb{Z}^+}$ of $M$ with the following properties: (1) $U_a'\Subset U_a$ for all $a$, (2) $\{U_a\}_{a\in\mathbb{Z}^+}$ is locally finite, (3) each $U_a$ is endowed with an open embedding $\varphi_a:U_a\to \mathbb{R}^n$, such that $\varphi_b\circ \varphi_a^{-1}: \varphi_a(U_a\cap U_b)\to\mathbb{R}^n$ is a $\overline{\mathcal{C}}$ embedding for each pair $a,b$. We show that one can isotope $\{\varphi_a\}$ and shrink the domains so that the transition functions are  $\mathcal{C}$ embeddings.

We construct a sequence of compactly supported isotopies $F_a:U_a\times I\to U_a$ such that the following conditions hold:
\begin{enumerate}
    \item $F_a(-,0)=\id$. 
    \item $\varphi_a\circ F_a(-,t)\circ \varphi_a^{-1}:\varphi_a(U_a)\to \varphi_a(U_a)$ is $\overline{\mathcal{C}}$ for all $t$. 
    \item Let $\varphi_a' = \varphi_a\circ F_a(-,1)$. Then there exists a sequence of open sets $\{\hat U_a\}_{a\in\mathbb{Z}^+}$ with $U_a'\Subset \hat{U}_a\Subset U_a$ such that $\varphi_b'\circ (\varphi_a')^{-1}: \varphi_a'(\hat{U}_a\cap \hat{U}_b)\to\mathbb{R}^n$ is a $\mathcal{C}$ embedding for each pair $a,b$.
\end{enumerate}

We construct the desired sequence by induction. Suppose we have $\{F_a\}$ for $a=1,\dots,q$ such that the above properties hold for all $a,b\in\{1,\dots,q\}$. Let $\hat{U}_a$ $(a=1,\dots,q)$ be the sequence of open sets given by Condition (3) above. 

Let $\hat{U}_{q+1}, \hat{U}_{q+1}'$ be open sets such that $U_{q+1}'\Subset \hat{U}_{q+1}'\Subset \hat{U}_{q+1} \Subset U_{q+1}$. For each $a=1,\dots,q$, let $\hat{U}_a'$ be an open set such that $U_a'\Subset \hat{U}_a'\Subset \hat{U}_a$. 

By Proposition \ref{prop_prop_C_smoothing_Cbar_embeddings}, for each $a=1,\dots,q$ and $\epsilon>0$, there is a compactly supported $\epsilon$--isotopy 
\[
G:\varphi_{q+1}(U_{q+1})\times I\to \varphi_{q+1}(U_{q+1})
\]
such that $G(-,0)=\id$, and $G(-,t)$ is $\overline{\mathcal{C}}$ for all $t$, and 
\[
\varphi_a' \circ \varphi_{q+1}^{-1}\circ G(-,1): \varphi_{q+1}(U_{q+1}\cap U_a)\to \mathbb{R}^n
\]
is a $\mathcal{C}$ embedding on $\varphi_{q+1}(\hat{U}_{q+1}\cap \hat{U}_a)$. Take $\epsilon$ such that
\[
\epsilon< \dist(\varphi_{q+1}(\hat{U}'_{q+1}\cap \hat{U}'_a),\mathbb{R}^n\setminus\varphi_{q+1}(\hat{U}_{q+1}\cap \hat{U}_a)).
\]
Define 
\[
F_{q+1,a}(-,t) = \varphi_{q+1}^{-1}\circ G(-,t)^{-1}\circ \varphi_{q+1}:U_{q+1}\to U_{q+1},
\]
\[
\varphi_{q+1,a}' = \varphi_{q+1}\circ F(-,1) = G(-,1)^{-1}\circ \varphi_{q+1}:U_{q+1}\to \mathbb{R}^n.
\]
Then 
\[
\varphi_a'\circ (\varphi_{q+1,a}')^{-1} = \varphi_a'\circ \varphi_{q+1}^{-1}\circ G(-,1):\varphi_{q+1,a}'(U_{q+1}\cap U_a)\to \mathbb{R}^n
\]
is a $\mathcal{C}$ embedding on $\varphi_{q+1,a}'(\hat{U}_{q+1}'\cap \hat{U}_a')$. This shows that applying the isotopy $F_{q+1,a}$ on $U_{q+1}$ yields a coordinate map $\varphi_{q+1,a}'$ such that Condition (3) above holds for $(q+1,a)$. 

Define $F_{q+1}$ by applying the isotopies given by $F_{q+1,1},F_{q+1,2},\dots,F_{q+1,q}$ in order. Moreover, when $a>1$, we take $B=\varphi_{q+1}(\hat{U}_a\cap (\hat{U}_1\cup\dots\cup\hat{U}_{a-1}))$ when invoking Proposition \ref{prop_prop_C_smoothing_Cbar_embeddings} in the construction of $F_{q+1,a}$, so that the $\mathcal{C}$--smoothness properties from the previous steps are preserved after shrinking the domains. This defines the desired $F_{q+1}$. 

Since $\{U_a\}_{a\in\mathbb{Z}^+}$ is locally finite, each $\hat{U}_a$ is only shrunk  finitely many times in the construction of  $\{F_a\}_{a\in\mathbb{Z}^+}$. So this sequence satisfies the desired properties.

The set $\{(\hat U_a,\varphi'_a)\}$ gives a $\overline{\mathcal{C}}$ chart of $M$ whose transition functions are $\mathcal{C}$ embeddings. 
\end{proof}

A similar result establishes the uniqueness of the $\mathcal{C}$ structure associated with a $\overline{\mathcal{C}}$ manifold.
\begin{Theorem}
\label{thm_uniqueness_of_C_structure}
Suppose $f:M\to N$ is a $\overline{\mathcal{C}}$ homeomorphism of $\mathcal{C}$ manifolds, and suppose $M,N$ are endowed with distance functions that are compatible with their topologies. Then for each $\epsilon>0$, there exists an $\epsilon$--isotopy $F:M\times I\to N$, such that $F(-,0)=f$, $F(-,1)$ is a $\mathcal{C}$ homeomorphism, and $F(-,t)$ is a $\overline{\mathcal{C}}$ homeomorphism for each $t$.
\end{Theorem}
\begin{proof}[Proof, assuming Proposition \ref{prop_prop_C_smoothing_Cbar_embeddings}]
The argument is similar to the proof of Theorem \ref{thm_construct_C_chart}. 
Take a countable, locally finite, $\mathcal{C}$--chart $\{U_a\}_{a\in\mathbb{Z}^+}$ of $M$ and take a $\mathcal{C}$--chart $\{V_b\}_{b\in\mathbb{Z}^+}$ of $N$, such that each $f(U_a)$ is included in $V_b$ for some $b$. Take $\{U_a'\}_{a\in\mathbb{Z}^+}$ such that $U_a'\Subset U_a$ and $\{U_a'\}$ covers $M$. By Proposition \ref{prop_prop_C_smoothing_Cbar_embeddings}, one can isotope the map $f$ on each $U_a$, so that after shrinking $U_a$ to some $\hat{U}_a$ with $U_a'\Subset \hat{U}_a \Subset U_a$, the resulting map is a $\mathcal{C}$--embedding of $\hat{U}_a$ into $N$. One can also ensure that the isotopy does not affect the $\mathcal{C}$--smoothness from the previous steps, after possibly shrinking the domain. The resulting isotopy then satisfies the desired properties.
\end{proof}

\subsection{Local contractibility}

A crucial input into our arguments is the following local contractibility result of Sullivan.  We will need to both use it and prove a variant of it, so we will be revisiting the proof later.

\begin{Theorem}
[\cite{sullivan1979hyperbolic}, see also {\cite[Theorem 3.4]{tukia1981lipschitz}}]
\label{thm_contract_C_embedding}
Suppose $U$ is an open subset of $\mathbb{R}^n$, and $A,A',B,B'$ are open subsets of $U$ such that $A' \Subset A$, $B'\Subset B$. Then for each $\epsilon>0$, there exists $\delta>0$ such that the following holds. Let $\mathcal{P}$ be the set of elements $f\in \Emb(U;\mathbb{R}^n)$ such that $\|f|_A-\id|_A\|_{C^0}<\delta$ and $f|_B$ is a $\mathcal{C}$ embedding. Then there exists a map 
    \[
    \mathcal{F}: \mathcal{P}\times I \to \Emb(U;\mathbb{R}^n),
    \]
    such that the following statements hold for each $f\in \mathcal{P}$. 
    \begin{enumerate}
        \item $\mathcal{F}(f,-)$ is an $\epsilon$--isotopy.
        \item The isotopy $\mathcal{F}(f,-)$ is compactly supported in $A$, and $\mathcal{F}(f,0) = f$.
        \item $\mathcal{F}(f,1)=\id$ on $A'$.
        \item $\mathcal{F}(f,t)|_{B'}$ is a $\mathcal{C}$ embedding for all $t$. 
        \item $\mathcal{F}(\id,t)=\id$ for all $t$.
    \end{enumerate}
\end{Theorem}

To prove Proposition~\ref{prop_prop_C_smoothing_Cbar_embeddings}, the idea is then to build on this to construct smoothings of $\overline{\mathcal{C}}$ embeddings to $\mathcal{C}$ embeddings. The overall idea is that if a $\overline{\mathcal{C}}$ embedding $f$ is $C^0$--close to $\id$, then the above local contractibility result yields an isotopy of $f$ to $\id$ (after possibly shrinking the domain), which is of course a $\mathcal{C}$ embedding. In general, we locally approximate $f$ in $C^0$ by 
$\mathcal{C}$ embeddings $g$, and apply Theorem \ref{thm_contract_C_embedding} to $g^{-1}\circ f$ on its natural domain. The difficulty is that to ensure compatibility, one should perturb step by step; however, after a perturbation, it is a priori not clear how the $\overline{\mathcal{C}}$ property is affected, due to the difficulty mentioned in the introduction that being globally $\overline{\mathcal{C}}$ is not clearly a local condition.  Resolving this is the content of the next section.

\section{$\overline{\mathcal{C}}$ implies globally $\overline{\mathcal{C}}$}
\label{sec_local_Cbar_to_global_Cbar}

We can now give the promised proof of Proposition~\ref{prop_prop_C_smoothing_Cbar_embeddings}; as we explained in the previous section, this implies the main theorems.  The proof of Proposition~\ref{prop_prop_C_smoothing_Cbar_embeddings} is the heart of the argument and will take up the remainder of the paper.   As we will see, the crux is to show
that if $f:U\to \mathbb{R}^n$ is a $\overline{\mathcal{C}}$ embedding and $V\Subset U$, then $f$ is globally $\overline{\mathcal{C}}$ on $V$.

The following concept will be helpful.  Recall that if $U\subset \mathbb{R}^n$ is an open subset and $F:U\times I \to \mathbb{R}^n$ is a compactly supported isotopy, then by Lemma \ref{lem_extend_emb_cpt_supported_difference}, we know that the images $F(U,t)$ are the same for all $t$. Let $V=F(U,t)$, then each $F(-,t)$ is a homeomorphism from $U$ to $V$. We introduce the following definition.

\begin{Definition}
\label{def_universal_Cbar_isotopy}
    Let $F:U\times I\to \mathbb{R}^n$ be a compactly supported isotopy, let $V=F(U,t)$. We say that $F$ is a \emph{universal $\overline{\mathcal{C}}$ isotopy}, if 
    \[
    F(-,t)\circ F(-,0)^{-1}:V\to V
    \]
    is globally $\overline{\mathcal{C}}$ on $V$ for all $t$. 
\end{Definition}

The concept of universal $\overline{\mathcal{C}}$ isotopy allows us to control global $\overline{\mathcal{C}}$ properties under the kind of local isotopies we will want to apply.

Our goal will then be to prove the following two results:

\begin{Proposition}
\label{prop_C_smoothing_in_global_Cbar_chart}
Suppose $U$ is an open subset of $\mathbb{R}^n$, and $A,A',B,B'$ are open subsets of $U$ such that $A' \Subset A$, $B'\Subset B$. Suppose $f:U\to \mathbb{R}^n$ is an embedding such that $f|_{B}$ is a $\mathcal{C}$ embedding and $f|_{A}$ is globally $\overline{\mathcal{C}}$.  Then for every $\epsilon>0$, there exists an $\epsilon$--isotopy $F:U\times I \to \mathbb{R}^n$, such that the following statements hold.
    \begin{enumerate}
        \item $F(-,0)=f$
        \item $F(-,1)$ is a $\mathcal{C}$ embedding on $A'\cup B'$.
        \item $F$ is compactly supported in $A$ and is a universal $\overline{\mathcal{C}}$ isotopy.
    \end{enumerate}
\end{Proposition}

\begin{Theorem}
\label{thm_local_Cbar_to_global_Cbar}
Suppose $V\Subset U$ are open subsets of $\mathbb{R}^n$.
    If $f$ is $\overline{\mathcal{C}}$ on $U$, then $f$ is globally $\overline{\mathcal{C}}$ on $V$.
\end{Theorem}

These clearly imply Proposition~\ref{prop_prop_C_smoothing_Cbar_embeddings}, by the following simple argument:
\begin{proof}[Proof of Proposition~\ref{prop_prop_C_smoothing_Cbar_embeddings}, assuming Theorem~\ref{thm_local_Cbar_to_global_Cbar} and Proposition~\ref{prop_C_smoothing_in_global_Cbar_chart}]
Let $U$, $A$, $A'$, $B$, $B'$, $f$ be as in Proposition~\ref{prop_prop_C_smoothing_Cbar_embeddings}.
After shrinking $A$, we may assume without loss of generality that $A\Subset U$.
By Theorem \ref{thm_local_Cbar_to_global_Cbar}, $f$ is globally $\overline{\mathcal{C}}$ on $A$. 
By Proposition \ref{prop_C_smoothing_in_global_Cbar_chart}, for each $\epsilon>0$, there exists an $\epsilon$--isotopy $F:U\times I\to \mathbb{R}^n$ such that $F(-,0)=f$, $F(-,1)$ is a $\mathcal{C}$ embedding on $A'\cup B'$, $F$ is compactly supported in $A$, and $F$ is a universal $\overline{\mathcal{C}}$ isotopy. Since $f$ is a $\overline{\mathcal{C}}$ embedding of $U$ and $F$ is a universal $\overline{\mathcal{C}}$ isotopy, we know that $F(-,t)$ is a $\overline{\mathcal{C}}$ embedding for each $t$. So $F$ satisfies the desired properties.
\end{proof}

\subsection{Properties of universal $\overline{\mathcal{C}}$ isotopies}

To proceed, we first collect some basic properties of $\overline{\mathcal{C}}$ isotopies.

The first lemma states that universal $\overline{\mathcal{C}}$ isotopies are invariant under globally $\overline{\mathcal{C}}$ coordinate changes.  

\begin{Lemma}
\label{lem_coord_change_universal_Cbar_isotopy}
    Suppose $U,V,W$ are open subsets of $\mathbb{R}^n$, and $h:V\to W$ is a globally $\overline{\mathcal{C}}$ homeomorphism. Suppose $F:U\times I\to V$ is a universal $\overline{\mathcal{C}}$ isotopy, then $F(U,t)$ is independent of $t$, and we further assume that $F(U,t)\Subset V$. Let $G:U\times I\to W$ be defined by $G=h\circ F$. Then $G$ is also a universal $\overline{\mathcal{C}}$ isotopy.
\end{Lemma}

\begin{proof}
    This follows immediately from the definitions and the observation that compositions of globally $\overline{\mathcal{C}}$ embeddings of compact subsets of $\mathbb{R}^n$ are also globally $\overline{\mathcal{C}}$. 
\end{proof}

Before proceeding to the further properties we will need, we establish the following two preliminaries.

\begin{Lemma}
\label{lem_C_approx_id}
    Suppose $U\subset \mathbb{R}^n$ is open, and $f:U\to \mathbb{R}^n$ is a globally $\overline{\mathcal{C}}$ embedding. Suppose $A\Subset U\setminus \supp f$. Then $f$ can be written as the $C^0$--limit of a sequence of $\mathcal{C}$ embeddings of $U$ that are equal to the identity on $A$. 
\end{Lemma}

\begin{proof}
    Let $V$, $V'$ be open subsets of $U$ such that $A\Subset V'\Subset V\Subset U\setminus \supp f$. Let $f_i:U\to \mathbb{R}^n$ be a sequence of $\mathcal{C}$ embeddings that converges to $f$ in $C^0$. Then $f_i|_{V}:V\to \mathbb{R}^n$ is a sequence of $\mathcal{C}$ embeddings that converges to $\id$. Applying Theorem \ref{thm_contract_C_embedding} with $(A,A') = (V',A)$ and $(B,B')=(V,V')$, we know that for each integer $m>0$, there exists $i_m$, such that for all $i>i_m$, there exists an $(1/m)$--isotopy from $f_i$ to some $g_i$, such that the isotopy is compactly supported in $V'$, and $g_i$ is a $\mathcal{C}$ embedding that equals the identity on 
    $A$.
    Let $j_m$ be an index such that $j_m>i_m$ and $\|f_{j_m} - f\|_{C^0(U)}<1/m$. Then $\{{g}_{j_m}\}_{m\ge 1}$ is a sequence of $\mathcal{C}$ embeddings that are equal to the identity on $A$ such that $\|{g}_{j_m}-f\|_{C^0(U)}<2/m$. Hence the result is proved. 
\end{proof}

\begin{Corollary}
\label{cor_extend_cpt_supported_Cbar_homeo}
    Suppose $U\subset \mathbb{R}^n$ is open, $f$ is a globally $\overline{\mathcal{C}}$ embedding of $U$ in $\mathbb{R}^n$ with compact support. Let $\hat{f}:\mathbb{R}^n\to \mathbb{R}^n$ be the extension of $f$ to $\mathbb{R}^n$ by the identity. Then:
    \begin{enumerate}
        \item $\hat{f}$ is globally $\overline{\mathcal{C}}$.
        \item $\hat{f}^{-1}$ is globally $\overline{\mathcal{C}}$.
    \end{enumerate}
\end{Corollary}

\begin{proof}
    Let $V',V$ be open subsets of $U$ such that $\supp f\subset V'\Subset V\Subset U$. Then by Lemma \ref{lem_C_approx_id}, there exists a sequence of embeddings $f_i:U\to \mathbb{R}^n$ such that $f_i\to f$ and $f_i=\id$ on $V\setminus V'$. Let $\hat{f}_i$ be the extension of $f_i|_{V}$ to $\mathbb{R}^n$ by the identity, then by Lemma \ref{lem_extend_cpt_supported_homeo}, the maps $\hat{f}_i$ are $\mathcal{C}$ homeomorphisms of $\mathbb{R}^n$ that $C^0$--converge to $\hat f$. This proves Part (1). By Lemma \ref{lem_inverse_top_embedding}, we know that $\hat{f}^{-1}$ is globally $\overline{\mathcal{C}}$ over every compact subset of $\mathbb{R}^n$. Since $\hat f$ is also compactly supported, applying Part (1) shows that $\hat{f}^{-1}$ is globally $\overline{\mathcal{C}}$ over $\mathbb{R}^n$.
\end{proof}

Now we can show that universal $\overline{\mathcal{C}}$ isotopies are
invariant under extensions of the domain:

\begin{Lemma}
\label{lem_universal_Cbar_extend_domain}
       Suppose  $F:U\times I\to \mathbb{R}^n$ is a
       compactly supported
       isotopy, 
       $W$ is an open subset of $\mathbb{R}^n$ that contains $U$, and the open embedding $g:W\to \mathbb{R}^n$ is an extension of $F(-,0)$. Define $\hat{F}(-,t)$ to be the extension of $F(-,t)$ to $W$ by $g$. Then $\hat{F}$ is a universal $\overline{\mathcal{C}}$ isotopy if and only if $F$ is a universal $\overline{\mathcal{C}}$ isotopy.
\end{Lemma}
\begin{proof}
    The fact that $\hat{F}$ being universal $\overline{\mathcal{C}}$ implies $F$ being universal $\overline{\mathcal{C}}$ follows from the definition. 
    To prove the other direction, note that $\hat{F}(-,t)\circ \hat{F}(-,0)^{-1}$ is the extension of $F(-,t)\circ F(-,0)^{-1}$ by the identity, so the desired result follows from Corollary \ref{cor_extend_cpt_supported_Cbar_homeo}.
\end{proof}

Here is another extension result we will need:

\begin{Lemma}
\label{lem_extend_universal_Cbar_isotopy}
    Suppose  $F:U\times I\to \mathbb{R}^n$ is a universal $\overline{\mathcal{C}}$ isotopy, and $W$ is an open subset of $\mathbb{R}^n$ such that $U\subset W$. Suppose there exists an open embedding $g:W\to \mathbb{R}^n$ which is an extension of $F(-,0)$, and suppose there is an open set $W'\subset W$ such that $g$ is globally $\overline{\mathcal{C}}$ on $W'$. Then, for each $t$, the extension of $F(-,t)$ to $W$ by $g$ is globally $\overline{\mathcal{C}}$ on $W'$. 
\end{Lemma}
\begin{proof}
    Let $H_t$ be the extension of $F(-,t)\circ F(-,0)^{-1}$ by the identity. By the assumptions and Lemma \ref{cor_extend_cpt_supported_Cbar_homeo}, we know that $H_t$ is a globally $\overline{\mathcal{C}}$ homeomorphism of $\mathbb{R}^n$. Let $g_t$ be the extension of $F(-,t)$ to $W$ by $g$. Then $g_t =  H_t\circ g$. So $g_t$ is globally  $\overline{\mathcal{C}}$ on $W'$  if $g$ is globally $\overline{\mathcal{C}}$ on $W'$.
\end{proof}

We also need the following result about compositions of universal $\overline{\mathcal{C}}$ isotopies.

\begin{Lemma}
\label{lem_compose_cptly_supported_isotopy}
     Suppose  $F_1,F_2:U\times I\to \mathbb{R}^n$ are compactly supported isotopies such that $F_1(-,1) = F_2(-,0)$, and let $F:U\times I \to \mathbb{R}^n$ be the concatenation of $F_1$ and $F_2$:
     \[
     F(x,t) = \begin{cases}
         F_1(x,2t) \quad &\text{ if } t\in [0,1/2]\\
         F_2(x,2t-1) \quad  &\text{ if } t\in [1/2,1].\\
     \end{cases}
     \]
If both $F_1$ and $F_2$ are universal $\overline{\mathcal{C}}$, then $F$ is universal $\overline{\mathcal{C}}$.
\end{Lemma}

\begin{proof}
   We have 
    \begin{align*}
    &F(-,t)\circ F(-,0)^{-1}=\\
    &\,\,\begin{cases}
        F_1(-,2t)\circ F_1(-,0)^{-1} &\text{ if }t\in[0,1/2],\\
        (F_2(-,2t-1)\circ F_2(-,0)^{-1})\circ (F_1(-,1)\circ F_1(-,0)^{-1}) &\text{ if }t\in[1/2,1].
    \end{cases}
    \end{align*}
    Therefore, if both $F_1$ and $F_2$ are universal $\overline{\mathcal{C}}$, then $F(-,t)\circ F(-,0)^{-1}$ is globally $\overline{\mathcal{C}}$ for each $t$. 
\end{proof}

\subsection{Local contractibility revisited}

To complete the proof, we will need the following refinement of Theorem \ref{thm_contract_C_embedding}.

\begin{Proposition}
\label{prop_match_global_contractibility}
The map $\mathcal{P}$ in Theorem \ref{thm_contract_C_embedding} can be chosen so that the following additional condition hold: If $f\in \mathcal{P}$ satisfies $f|_A$ is globally $\overline{\mathcal{C}}$, then $\mathcal{F}(f,-)$ is a universal $\overline{\mathcal{C}}$ isotopy.
\end{Proposition}

We will defer the proof, which requires saying more about Sullivan's hyperbolic torus trick, to Section \ref{subsec_adapt_local_contract}.

\subsection{Proofs of the main results}

We can now give the promised proofs of the two main results of this section, assuming Proposition \ref{prop_match_global_contractibility}. We repeat the statements of the results below.

\begin{repProposition}{prop_C_smoothing_in_global_Cbar_chart}
Suppose $U$ is an open subset of $\mathbb{R}^n$, and $A,A',B,B'$ are open subsets of $U$ such that $A' \Subset A$, $B'\Subset B$. Suppose $f:U\to \mathbb{R}^n$ is an embedding such that $f|_{B}$ is a $\mathcal{C}$ embedding and $f|_{A}$ is globally $\overline{\mathcal{C}}$.  Then for every $\epsilon>0$, there exists an $\epsilon$--isotopy $F:U\times I \to \mathbb{R}^n$, such that the following statements hold.
    \begin{enumerate}
        \item $F(-,0)=f$
        \item $F(-,1)$ is a $\mathcal{C}$ embedding on $A'\cup B'$.
        \item $F$ is compactly supported in $A$ and is a universal $\overline{\mathcal{C}}$ isotopy.
    \end{enumerate}
\end{repProposition}

\begin{proof}
    Let $\hat A,\hat{A}'$ be open sets such that $A'\Subset \hat{A}'\Subset \hat A\Subset A$. Let $g_i:A\to \mathbb{R}^n$ be a sequence of $\mathcal{C}$ embeddings such that $g_i\to f|_A$ in $C^0$. Then by Lemma \ref{lem_inverse_top_embedding}, we have $g_i(A)\supset f(\hat A)$ for $i$ sufficiently large, and $h_i = g_i^{-1} \circ f|_{\hat A}$ converges to $\id|_{\hat A}$ in $C^0$. Therefore, when $i$ is sufficiently large, we have $\Ima(h_i)\subset A$.

    We will want to isotope an $h_i$ to produce $F$, and we want $F$ to be an $\epsilon$-isotopy.  Towards proving the $\epsilon$--isotopy property,
    we first extract an auxiliary parameter $\delta > 0$ from $\epsilon$ as follows.  Let $i_0$ be an index such that for all $i\ge i_0$, we have $\Ima(h_i)\subset A$
    and $\|g_i-f\|_{C^0(A)}<\epsilon/3$. After shrinking $A$ if necessary, we may assume without loss of generality that $f$ is uniformly continuous on $A$.  Now choose  $\delta>0$ such that for all $x,y\in  A$ with $|x-y|<\delta$, we have $|f(x)-f(y)|<\epsilon/3$. Note that if $i\ge i_0$,
    and $x,y\in  A$ satisfy $|x-y|<\delta$, then $|g_i(x)-g_i(y)|<\epsilon$. 
    
    Now given the above $\delta > 0$, choose $i\ge i_0$ sufficiently large such that we may invoke Theorem \ref{thm_contract_C_embedding} to find a $\delta$--isotopy $G:\hat{A}\times I\to \mathbb{R}^n$ that is compactly supported in $\hat A'$, such that $G(-,0)=h_i$, and $G(-,1)$ is equal to the identity on $A'$, and $G(-,t)$ a $\mathcal{C}$ embedding on $B'\cap \hat{A}'$ for all $t$.  By Proposition \ref{prop_match_global_contractibility}, the isotopy $G$ can be taken to be universal $\overline{\mathcal{C}}$, since $h_i$ is a globally $\overline{\mathcal{C}}$ embedding of $\hat A$.  By the assumptions on $i_0$, we know that the image of $G$ is contained in $A$, so $g_i\circ G$ is a well-defined isotopy with compact support, and $g_i\circ G(-,0)=f|_{\hat{A}}$.    
    Define $F(-,t)$ to be the extension of $g_i\circ G(-,t)$ to $U$ by $f$.

    We now show that $F$ satisfies all of the properties required by the proposition.
    
    To start, since
    $G$ is a $\delta$--isotopy, the definition of $\delta$ implies that $g_i\circ H$ is an $\epsilon$--isotopy.   It is also immediate that $F(-,0)=f$  and that $F$ is compactly supported in $A$

    To show that $F$ is universal $\overline{\mathcal{C}}$, we note that because $G$ is universal $\overline{\mathcal{C}}$, Lemma \ref{lem_coord_change_universal_Cbar_isotopy} implies that $g_i\circ G$ is universal $\overline{\mathcal{C}}$.  Hence, by Lemma \ref{lem_universal_Cbar_extend_domain},  $F$ is also universal $\overline{\mathcal{C}}$.

    Finally, it remains to show item (2).  We know by construction that  $G(-,1)$ is a $\mathcal{C}$ embedding on $A'\cup (B'\cap \hat A)$.
    Hence, $g_i \circ G(-,1)$ is a $\mathcal{C}$ embedding on $A'\cup (B'\cap \hat A)$.  By assumption, $f$ is $\mathcal{C}$ on $B$.  Thus, since $g_i \circ G(-,1)$ only differs from $f$  in a compact subset of $\hat A$,  $F(-,1)$ is a $\mathcal{C}$ embedding on $A'\cup B'$.
\end{proof}

\begin{repTheorem}{thm_local_Cbar_to_global_Cbar}
Suppose $V\Subset U$ are open subsets of $\mathbb{R}^n$.
    If $f$ is $\overline{\mathcal{C}}$ on $U$, then $f$ is globally $\overline{\mathcal{C}}$ on $V$.
\end{repTheorem}

\begin{proof}
    Let $\{U_\alpha\}_{\alpha\in S}$, $\{U_\alpha'\}_{\alpha\in S}$ be open coverings of $U$ such that $U_\alpha'\Subset U_\alpha$ and $f$ is globally $\overline{\mathcal{C}}$ on each $U_\alpha$. Then there is a finite subset $S'\subset S$ such that $\{U_\alpha'\}_{\alpha\in S'}$ covers $V$. 
    Without loss of generality, write $S'=\{1,2,\dots,p\}$.

    We use induction to show the following claim: for each $q\in \{1,\dots,p\}$ and each $\epsilon>0$, there exists an $\epsilon$--isotopy $F:U\times I\to \mathbb{R}^n$ and open sets $\hat{U}_\alpha$ with $U_\alpha'\Subset \hat{U}_\alpha\Subset U_\alpha$ $(\alpha=1,\dots,q)$, such that 
    \begin{enumerate}
        \item $F(-,0)=f$.
        \item $F$ is a universal $\overline{\mathcal{C}}$ isotopy compactly supported in $\cup_{\alpha=1}^q U_\alpha$.
        \item $F(-,1)$ is a $\mathcal{C}$ embedding on $\cup_{\alpha=1}^q \hat U_\alpha$. 
    \end{enumerate}
    
    The case for $q=1$ follows immediately from Proposition \ref{prop_C_smoothing_in_global_Cbar_chart}. Suppose the result holds for $q=q_0<p$, we show that it also holds for $q=q_0+1$. Suppose $F$, ${\hat U}_\alpha$ $(\alpha=1,\dots,q_0)$ satisfy the above properties with respect to $q=q_0$ and $\epsilon = \epsilon_0$.
    
    Now let ${\hat U}_\alpha'$ $(\alpha=1,\dots,q_0)$ be open sets such that $U_\alpha'\Subset {\hat U}_\alpha'\Subset \hat{U}_\alpha$ $(\alpha=1,\dots,q_0)$. Let $\hat U_{q_0+1}$, $\hat U_{q_0+1}'$ be open sets such that $U_{q_0+1}'\Subset \hat U_{q_0+1}'\Subset \hat U_{q_0+1}\Subset U_{q_0+1}$. 
    Since $F$ is universal $\overline{\mathcal{C}}$, by Lemma \ref{lem_extend_universal_Cbar_isotopy}, the map $F(-,1)$ is globally $\overline{\mathcal{C}}$ on each $U_\alpha$.  
    Therefore, we can apply Proposition \ref{prop_C_smoothing_in_global_Cbar_chart} with $\epsilon = \epsilon_0$, $f= F(-,1)$, $A=\hat{U}_{q_0+1}$, $A'=\hat U_{q_0+1}'$, $B=\cup_{q=1}^{q_0} \hat{U}_q$, $B'=\cup_{q=1}^{q_0} \hat{U}_q'$. This yields an $\epsilon_0$--isotopy $G$ such that $G(-,0)=f= F(-,1)$, and $G$ is a universal $\overline{\mathcal{C}}$ isotopy supported in $\hat{U}_{q_0+1}$, and $G(-,1)$ is a $\mathcal{C}$ embedding on $\cup_{\alpha = 1}^{q_0+1} \hat U_\alpha'$. By Lemma \ref{lem_compose_cptly_supported_isotopy}, the concatenation of $F$ with $G$ is also a universal $\overline{\mathcal{C}}$ isotopy, and it satisfies the desired properties for $q=q_0+1$ and $\epsilon = 2\epsilon_0$. Since $\epsilon_0>0$ can be taken arbitrarily, the claim is proved. 

    Taking $\epsilon\to 0$, the claim implies that $f|_V$ can be $C^0$ approximated by $\mathcal{C}$ embeddings. Therefore, $f$ is globally $\overline{\mathcal{C}}$ on $V$.
\end{proof}

\subsection{The adaptation of local contractibility}
\label{subsec_adapt_local_contract}
It remains to explain the proof of Proposition \ref{prop_match_global_contractibility}. To prove Proposition \ref{prop_match_global_contractibility}, we have to say more about how Sullivan's local contractibility statement is proved.    We now review what we need to know. 

\subsubsection{Structure of Sullivan's argument}

At the heart of the argument is the following {\em handle straightening theorem} due to Sullivan, building on works of Edwards--Kirby \cite{edwards1971deformations,kirby1969stable}.  Recall that if $k+m=n$, the \emph{$n$--dimensional $k$--handle} is the pair $(H,A)= (B^m\times B^k,B^m\times \partial B^k)$, and the \emph{core} of the handle $H$ is defined to be $\{0\}\times B^k$.  The handle straightening theorem isotopes a compactly supported embedding of a handle near the identity to be the identity near the core while preserving the $\mathcal{C}$--condition.  More precisely, we have the following:

\begin{Theorem}[\cite{sullivan1979hyperbolic}, see also {\cite[Theorem 3.3]{tukia1981lipschitz}}]
\label{thm_handle_straighten_in_C}
Suppose $m,k$ are non-negative integers such that $m+k=n$.
Let $\epsilon_0\in(0,1/2)$ be a constant.  
Then there exists an open neighborhood $\mathcal{P}$ of 
\[
\id \in \Emb(B^m\times B^k, B^m\times (B^k\setminus B^k(1-\epsilon_0));\mathbb{R}^n),
\]
a neighborhood $N$ of $\partial(B^m(1/2)\times B^k)$ in $B^m(1/2)\times B^k$, and a map 
 \[
 \mathcal{F}: \mathcal{P} \times I \to \Emb(B^m\times B^k;\mathbb{R}^n)
 \]
 such that the following statements hold for each $f\in\mathcal{P}$.
 \begin{enumerate}
     \item $\mathcal{F}(f,0)=f$ and $\mathcal{F}(f,1)$ equals the identity on $B^m(1/4)\times B^k$.
     \item The isotopy $\mathcal{F}(f,-)$ is compactly supported in $(B^m(1/2)\times B^k)\setminus N$.
     \item If $f=\id$, then $\mathcal{F}(f,t)=\id$ for all $t$. 
     \item If $f$ is a $\mathcal{C}$ embedding, then $\mathcal{F}(f,t)$ is a $\mathcal{C}$ embedding for all $t$. 
 \end{enumerate} 
\end{Theorem}

\begin{remark}
   If $k=0$, then by our convention, $B^k\setminus B^k(1-\epsilon_0)=\emptyset$. If $m=0$, our convention is $B^m(r) = B^m = \{0\}$ for all $r>0$, so last part of Statement (1) is equivalent to $\mathcal{F}(f,1)=\id$ on $B^n$. 
\end{remark}

Recall that if $M$ is a topological $n$--manifold, $(H,A)$ is an $n$ dimensional handle, and $\tau:A\to \partial M$ is an embedding, then the glued space $M\cup_\tau H$ is also a topological manifold, and it is called the manifold obtained by a \emph{$k$--handle attachment} via $\tau$. 
The handle straightening theorem 
implies that if $M\cup_\tau H$ is a submanifold of $\mathbb{R}^n$, and if $f:M\cup_\tau H\to \mathbb{R}^n$ is an embedding such that $f=\id$ on a neighborhood of $M$ and $f$ is sufficiently close to $\id$ on $H$, then one can modify the embedding by an isotopy supported in the interior of $H$ so that the result equals the identity on both $M$ and a neighborhood of the core of $H$. 

Theorem \ref{thm_handle_straighten_in_C} implies the local contractibility Theorem \ref{thm_contract_C_embedding} by the following argument:

\begin{proof}
[Proof of Theorem \ref{thm_contract_C_embedding} from Theorem \ref{thm_handle_straighten_in_C}]

Endow $A$ with a locally finite handlebody structure, and let $H\subset A$ be a finite sub-handlebody whose interior covers the closure of $A'$. We further require that if a handle $h$ of $H$ intersects $B'$, then there is a sub-handlebody of $H$ that contains $h$ and is contained in $B$. This can be achieved because there exists $r>0$ such that the $r$--neighborhood of $B'$ is contained in $B$, so the desired property is satisfied if the diameter of every handle is less than $r/(n+1)$. Applying Theorem \ref{thm_handle_straighten_in_C} to each handle of $H$ in the order of their indices yields a map $\mathcal{F}$ that satisfies Statements (2)-(5). If we further require that the diameters of all handles in $H$ are less than $\epsilon/(n+1)$, then the isotopy also satisfies Statement (1). 
\end{proof}

Sullivan's handle straightening theorem is itself essentially a consequence of the following ``extension after restriction" type result that we will also need as an input:

\begin{Theorem}[\cite{sullivan1979hyperbolic}, see also {\cite[Theorem 3.3]{tukia1981lipschitz}}]
\label{thm_match_handle_in_C}
Suppose $m,k$ are non-negative integers such that $m+k=n$.
Let $\epsilon_0\in(0,1/2)$ be a constant.    
Then there exists an open neighborhood $\mathcal{P}$ of 
\[
\id \in \Emb(B^m\times B^k, B^m\times (B^k\setminus B^k(1-\epsilon_0));\mathbb{R}^n),
\]
a neighborhood $N$ of $\partial(B^m(1/2)\times B^k)$ in $B^m(1/2)\times B^k$, and a map
    \[
\psi:\mathcal{P} \to \Emb(B^m(1/2)\times B^k, N; \mathbb{R}^n),
   \]
    such that 
   \begin{enumerate}
       \item $\psi(\id) = \id$.
       \item $f = \psi(f)$ on $B^m(1/4)\times B^k$ for all $f\in \mathcal{P}$.
        \item If $f\in \mathcal{P}$ is a $\mathcal{C}$ embedding, then $\psi(f)$ is a  $\mathcal{C}$ embedding. 
    \end{enumerate}
\end{Theorem}

Extension after restriction implies handle straightening by the following ``Alexander trick", which we now explain.

\begin{proof}[Proof of Theorem \ref{thm_handle_straighten_in_C} from Theorem \ref{thm_match_handle_in_C}]
    Let $\psi, \mathcal{P}$ be as in Theorem \ref{thm_match_handle_in_C} and let $f\in \mathcal{P}$. Since $\psi(f)$ is an embedding of $B^m(1/2)\times B^k$ that equals the identity near the boundary, it is a self-homeomorphism of $B^m(1/2)\times B^k$ by Lemma \ref{lem_extend_cpt_supported_homeo}. Let $\hat \psi(f)$ denote the extension of $\psi(f)$ to $\mathbb{R}^n$ by the identity. For $t\in(0,1]$, define 
    \begin{align*}
    J_t(f) : \mathbb{R}^n &\to \mathbb{R}^n \\
                        x &\mapsto t\cdot \hat{\psi}(f)(x/t),
    \end{align*}
    and define $J_0(f)=\id$. 
    By Lemma \ref{lem_extend_cpt_supported_homeo}, every $J_t(f)$ is a self homeomorphism of $\mathbb{R}^n$; if $\psi(f)$ is a $\mathcal{C}$ homeomorphism, then each $J_t(f)$ is a $\mathcal{C}$ homeomorphism. Define $\mathcal{F}(f,t) =  J_t(f)^{-1}\circ f$. 
    Note that $\{J_t(f)^{-1}\}$ is compactly supported in the interior of ${B}^m(1/2)\times {B}^k$. Therefore, by Lemma \ref{lem_inverse_top_embedding}, the support of $\mathcal{F}(f,-) = f^{-1}(\supp \{{J}_t(f)^{-1}\})$ is contained in a compact subset $A$ of the interior of ${B}^m(1/2)\times {B}^k$ when $f$ is sufficiently close to $\id$, where $A$ only depends on $\|f-\id\|_{C^0}$.  So $\mathcal{F}$ satisfies the desired properties. 
\end{proof}

We will review the proof of Theorem \ref{thm_match_handle_in_C} in Section \ref{sec_appendix}.

\subsubsection{Our refinement: extension after restriction revisited}

To prove our refined local contractibility result stated in Proposition \ref{prop_match_global_contractibility}, we first need to refine the extension-after-restriction statement.  Namely, we prove:

\begin{Proposition}
\label{prop_match_global_Cbar_handle}
   In Theorem \ref{thm_match_handle_in_C}, one can choose $\mathcal{P}$, $N$, $\psi$ to satisfy the following additional property: if $f\in\mathcal{P}$ is globally $\overline{\mathcal{C}}$, then $\psi(f)$ is globally $\overline{\mathcal{C}}$. 
\end{Proposition}

\begin{remark}
    By Corollary \ref{cor_extend_cpt_supported_Cbar_homeo}, if $\psi(f)$ is globally $\overline{\mathcal{C}}$, then the extension of $\psi(f)$ to $\mathbb{R}^n$ by the identity is a globally $\overline{\mathcal{C}}$ homeomorphism of $\mathbb{R}^n$. 
\end{remark}

\begin{proof}
    To simplify notation, we write $B^k(1)\setminus B^k(1-\epsilon_0)$ as $B^k((1-\epsilon_0,1])$, with the understanding that this set is empty if $k=0$.
    
    Let $\epsilon_0$ be the constant in the assumptions of Theorem \ref{thm_match_handle_in_C}, let $r \in (1-\frac{\epsilon_0}{2},1)$. Let $s_{r}:\mathbb{R}^n\to\mathbb{R}^n$ be the homeomorphism that takes $x\in\mathbb{R}^n$ to $rx$. By Lemma \ref{lem_C_approx_id}, if $f\in \Emb(B^m\times B^k, B^m\times B^k((1-\epsilon_0,1]);\mathbb{R}^n)$ is globally $\overline{\mathcal{C}}$, then $f|_{s_r(B^m\times B^k)}$ can be $C^0$ approximated by a sequence of $\mathcal{C}$ embeddings in 
    \[
    \Emb\Bigg(s_r(B^m\times B^k), s_r\Big(B^m\times B^k\big((\frac{1-\frac{\epsilon_0}{2}}{r},1]\big)\Big);\mathbb{R}^n\Bigg).
    \]

    Note that by a coordinate change of the $B^m$ factor, the radius $1/4$ in Statement (2) of Theorem \ref{thm_match_handle_in_C} can be changed to an arbitrary constant in $(0,1/2)$, and we change the radius to $1/(4r)$. 
    By rescaling, we may apply Theorem \ref{thm_match_handle_in_C} to embeddings of $s_r(B^m\times B^k)$. 
    Hence by Theorem \ref{thm_match_handle_in_C}, there exist a neighborhood $\mathcal{P}_r$ of 
    \[
    \id\in \Emb\Bigg(s_r(B^m\times B^k), s_r\Big(B^m\times B^k\big((\frac{1-\frac{\epsilon_0}{2}}{r},1]\big)\Big);\mathbb{R}^n\Bigg),
    \]
    a neighborhood $N$ of $\partial(B^m(1/2)\times B^k)$ in $B^m(1/2)\times B^k$, and a map 
    \[
    \psi_r: \mathcal{P}_r \to \Emb(s_r(B^m(1/2)\times B^k), s_r(N); \mathbb{R}^n)
    \]
    such that $\psi_r(f) = f$ on $s_r(B^m(1/(4r))\times B^k)= B^m(1/4)\times B^k(r)$. 

    Let $\mathcal{P}$ be the set of elements $f\in \Emb(B^m\times B^k, B^m\times B^k((1-\epsilon_0,1]);\mathbb{R}^n)$ such that $f|_{s_r(B^m\times B^k)}\in \mathcal{P}_r$. For $f\in \mathcal{P}$, let $f_r = f|_{s_r(B^m\times B^k)}$. Let $\psi(f)$ be the extension of $\psi_r(f_r)$ from $s_r(B^m(1/2)\times B^k)$ to $B^m(1/2)\times B^k$ by the identity. 
    Since $\psi_r(f_r)$ equals the identity near $\partial(s_r(B^m(1/2)\times B^k))$, we know that $\psi(f)$ is a homeomorphism that equals the identity near $\partial(B^m(1/2)\times B^k)$. It is also straightforward to verify that $\psi(f)$ satisfies all of the properties stated in Theorem \ref{thm_match_handle_in_C}.
    
    If $f$ is a globally $\overline{\mathcal{C}}$ embedding, then by Lemma \ref{lem_C_approx_id}, $f_r$ can be $C^0$ approximated by a sequence of $\mathcal{C}$ embeddings $g_i:s_r(B^m\times B^k)\to \mathbb{R}^n$ in $\mathcal{P}_r$. So $\psi_r(g_i)$ converge to $\psi_r(f_r)$ in $C^0$, which implies $\psi_r(f_r)$ is globally $\overline{\mathcal{C}}$ on $s_r(B^m(1/2)\times B^k)$. By Corollary \ref{cor_extend_cpt_supported_Cbar_homeo}, we conclude that $\psi(f)$ is globally $\overline{\mathcal{C}}$ on $B^m(1/2)\times B^k$.
\end{proof}

\subsubsection{Handle straightening revisited}

Now we can use the result from the previous section to prove a refinement of Sullivan's handle straightening.

Before stating and proving the refinement, we collect the following preliminary:

\begin{Lemma}
\label{lem_alex_isot_global_Cbar}
Suppose $f:B^n\to B^n$ is a globally $\overline{\mathcal{C}}$ homeomorphism that equals the identity near $\partial B^n$. Let $\hat f$ be the extension of $f$ to $\mathbb{R}^n$ by the identity. Consider the Alexander isotopy defined by
\begin{align}
   J_t(f):\mathbb{R}^n &\to \mathbb{R}^n \nonumber \\
        x &\mapsto t\cdot \hat{f}(x/t) \label{eqn_Alexander_isotopy}
\end{align}
 for $t\in (0,1]$, and $J_0(f)=\id$. Then $J_t(f)$ and $(J_t(f))^{-1}$ are both globally $\overline{\mathcal{C}}$ homeomorphisms for all $t$. 
\end{Lemma}

\begin{proof}
    This is an immediate consequence of Corollary \ref{cor_extend_cpt_supported_Cbar_homeo}.
\end{proof}

Now we give the promised statement and proof.

\begin{Proposition}
\label{prop_handle_straightening_in_Cbar}
In Theorem \ref{thm_handle_straighten_in_C}, one can choose $\mathcal{F}$ so that it satisfies the following additional property: if $f\in \mathcal{P}$ is globally $\overline{\mathcal{C}}$, then the isotopy $\mathcal{F}(f,-)$ is a universal $\overline{\mathcal{C}}$ isotopy.
\end{Proposition}

\begin{proof}
    Repeat the proof of Theorem \ref{thm_handle_straighten_in_C} from Theorem \ref{thm_match_handle_in_C}, with the neighborhood $\mathcal{P}$ and the map $\psi$ provided by Proposition \ref{prop_match_global_Cbar_handle}; let $\mathcal{F}$ be the resulting map. If $f$ is globally $\overline{\mathcal{C}}$, then $\psi(f)$ is globally $\overline{\mathcal{C}}$. 
    Let $J_t(\psi(f))$ $(t\in[0,1])$ be the Alexander isotopy of $\psi(f)$ as given by \eqref{eqn_Alexander_isotopy}.
    We have $\mathcal{F}(f,t)=(J_t(\psi(f)))^{-1}\circ f$. Hence
    $\mathcal{F}(f,t)\circ \mathcal{F}(f,0)^{-1} = J_t(\psi(f))^{-1}$. So $\mathcal{F}(f,-)$ 
    is a universal $\overline{\mathcal{C}}$ isotopy by Lemma \ref{lem_alex_isot_global_Cbar}.
\end{proof}

\subsubsection{Refined local contractibility }

Now we prove the refined contractibility result, which finishes the proofs of the main theorems. We repeat the statement of the proposition below.

\begin{repProposition}{prop_match_global_contractibility}
The map $\mathcal{P}$ in Theorem \ref{thm_contract_C_embedding} can be chosen so that the following additional condition holds: If $f\in \mathcal{P}$ satisfies $f|_A$ is globally $\overline{\mathcal{C}}$, then $\mathcal{F}(f,-)$ is a universal $\overline{\mathcal{C}}$ isotopy.
\end{repProposition}

\begin{proof}
    Repeat the proof of Theorem \ref{thm_contract_C_embedding} from Theorem \ref{thm_handle_straighten_in_C}, with the map $\mathcal{F}$ in Theorem \ref{thm_handle_straighten_in_C} provided by Proposition \ref{prop_handle_straightening_in_Cbar}. 
    
    We use induction to show that the isotopy in each step of the proof of Theorem \ref{thm_contract_C_embedding} is universal $\overline{\mathcal{C}}$ under this setup. By Lemma \ref{lem_compose_cptly_supported_isotopy}, if the first $k$ isotopies are all universal $\overline{\mathcal{C}}$, then $f$ is globally $\overline{\mathcal{C}}$ on $A$ after the first $k$ isotopies. Hence, by Proposition \ref{prop_handle_straightening_in_Cbar} and Lemma \ref{lem_universal_Cbar_extend_domain}, the isotopy in the $(k+1)^{th}$ step is also universal $\overline{\mathcal{C}}$. This finishes the induction argument.

    Therefore, the resulting isotopy is a composition of universal $\overline{\mathcal{C}}$ isotopies, which is itself universal $\overline{\mathcal{C}}$ by Lemma \ref{lem_compose_cptly_supported_isotopy}.
\end{proof}

\section{Appendix: More details of known arguments and further remarks}
\label{sec_appendix}
This section provides an exposition for the proof of the ``extension after restriction'' Theorem \ref{thm_match_handle_in_C}.
We include this for the ease of the reader and to make the argument more self-contained. 
In particular, we collect the details for arguments that we found difficult to locate in the literature in the form needed for our context.  
We also include some remarks we think are of value.

\subsection{Canonical Schoenflies}
    \label{subsec_canonical_schoenflies}

    To prove Theorem \ref{thm_match_handle_in_C},
    we need to show 
    that one may choose $\psi(f)$ that depends continuously on $f$ in the $C^0$ topology.
    This requires the following input, which can be regarded as a \emph{canonical} Schoenflies theorem in the category $\mathcal{C}$.  We give a reference and explain the proof for the convenience of the reader.  The proof is not actually referred to in subsequent arguments, so it can be skipped by an uninterested reader.
    
\begin{Theorem}[\cite{gauld1977quasiconformal}]
\label{thm_canonical_schoenflies}
    Let $\epsilon_0 \in (0,1/2)$.
    Then there exists an open neighborhood $\mathcal{P}$ of $\id\in \Emb(B^n([1-\epsilon_0,1]);\mathbb{R}^n)$ and a continuous map 
    \[
    \varphi: \mathcal{P}\to \Emb(B^n; \mathbb{R}^n),
    \]
    such that the following statements hold:
    \begin{enumerate}
        \item There exists $r\in (1-\epsilon_0, 1)$ such that for each $f\in \mathcal{P}$, we have $\varphi(f)|_{B^n([r,1])} = f|_{B^n([r,1])}$.
        \item $\varphi(\id)=\id$.
        \item $\varphi$ takes $\mathcal{C}$ embeddings to $\mathcal{C}$ embeddings.
    \end{enumerate}
\end{Theorem}

\begin{proof}
    This is a consequence of \cite[Lemma 9]{gauld1977quasiconformal}; see also \cite[Lemma 3.2]{tukia1981lipschitz}. Note that the notation $[a,b]$ in \cite{gauld1977quasiconformal} corresponds to our notation $B^n([a,b])$, and the open and half-open interval notations in \cite{gauld1977quasiconformal} are defined similarly. The notation $[a]$ in \cite{gauld1977quasiconformal} corresponds to our notation $\partial B^n(a)$. The notation $\overline{B}^n$ in \cite{gauld1977quasiconformal} corresponds to our notation $B^n$. 
    
    For $r\neq 0$, let $s_r:\mathbb{R}^n\to \mathbb{R}^n$ denote the homeomorphism that takes $x$ to $rx$. 
    Let $\mu:[0,1]\to[0,1]$ be a piecewise-linear homeomorphism such that
    \[
    \mu(0) = 0, \quad \mu(1/2) = 1-\epsilon_0/2, \quad \mu(1) = 1.
    \]
    Let $a,b,c$ be constants such that $c\in(1/2,1)$ and 
    \begin{equation}\label{eqn_canonical_schoeflies_abc}
    0<1-\epsilon_0^2/4<a<(1+\epsilon_0/2)\,\mu(c)<b<1.
    \end{equation}
    We also abuse notation and use $\mu$ to denote the homeomorphism
    \[\mu:B^n\to B^n\]
    defined by the property that if $r\in [0,1]$ and $x\in \partial B^n$, then $\mu(rx) = \mu(r)x$. 

    We note that it is possible to choose $a,b,c$ such that $b/a = 2c$. In fact, for every $c$ in $(1/2,\mu^{-1}(\frac{1}{1+\epsilon_0/2}))$, the possible values of $b/a$ for pairs $(a,b)$ satisfying \eqref{eqn_canonical_schoeflies_abc} form the interval $(1,\frac{1}{1-\epsilon_0^2/4})$. So there exist $a,b$ satisfying \eqref{eqn_canonical_schoeflies_abc} such that $2c=b/a$ if we take $c<\frac{1}{1-\epsilon_0^2/4}$. 
    
    Let $\mathcal{P}$ be the set consisting of all $f\in \Emb(B^n([1-\epsilon_0,1]);\mathbb{R}^n)$ such that the embedding $e= s_{1+\epsilon_0/2}\circ f\circ \mu: B^n([1/2,1])\to \mathbb{R}^m$ satisfies
    \[
        B^n([b,1]) \subset e(B^n((c,1))),\quad \partial B^n(a)\subset e(B^n((1/2,c))). 
    \]
    Then $\mathcal{P}$ is an open set that contains $\id$. By \cite[Lemma 9]{gauld1977quasiconformal}, there exists an embedding 
    $\hat{e}: B^n \to \mathbb{R}^n$
    that equals $e$ on $B^n([c,1])$ and is quasi-conformal when $e$ is quasi-conformal. 
    It is clear from the construction that $\hat e$ depends continuously on $e$ with respect to the $C^0$--topology.
    If we choose $a,b,c$ such that $2c=b/a$, then it is also straightforward to verify from the construction that $\hat e$ is bi-Lipschitz if $e$ is bi-Lipschitz. Define $\varphi_1(f) = s_{1+\epsilon_0/2}^{-1}\circ \hat e\circ \mu^{-1}:B^n\to \mathbb{R}^n$. Then $\varphi_1(f)$ is an embedding of $B^n$ that depends continuously on $f$ in the $C^0$--norm, and $\varphi_1(f) = f$ on $\mu(B^n([c,1]))$, and $\varphi_1(f)$ is a $\mathcal{C}$ embedding if $f$ is. By Lemma \ref{lem_extend_emb_cpt_supported_difference}, $\varphi_1(\id)$ is a homeomorphism of $B^n$. Hence $\varphi(f) = \varphi_1(f)\circ \varphi_1(\id)^{-1}$ satisfies the desired properties. 
\end{proof}

\begin{Corollary}
Let $\mathcal{P}$, $\varphi$ be as in Theorem \ref{thm_canonical_schoenflies}.
   If $f\in \mathcal{P}$ is a globally $\overline{\mathcal{C}}$ embedding, then $\varphi(f): B^n\to \mathbb{R}^n$ is a globally $\overline{\mathcal{C}}$ embedding.
\end{Corollary}

\begin{proof}
   Suppose $f_i\in \mathcal{P}$ are $\mathcal{C}$ embeddings that converge to $f$ in $C^0$, then $\varphi(f_i)$ are $\mathcal{C}$ embeddings that converge to $\varphi(f)$ in $C^0$.
\end{proof}

\subsection{Sullivan's arguments and the work of Kirby--Edwards}

We now explain Sullivan's proof of Theorem \ref{thm_match_handle_in_C}.  These arguments have antecedents in the work of Kirby--Edwards \cite{edwards1971deformations,kirby1969stable}, which we will also partly explain as context.

\subsubsection{A warm-up case and the torus trick}
\label{subsec_zero_handle_straightening}

To illustrate some of the main ideas, and in particular the famous ``torus trick" that we will need to build on, we first prove a weaker variant of ``extension after restriction''
due to Kirby \cite{kirby1969stable}.  Strictly speaking, we will not literally need this in future arguments, so it can be skipped by the uninterested reader; however, it gives useful context in understanding Sullivan's arguments and our adaptation of them.

\begin{Proposition}[\cite{kirby1969stable}]
\label{prop_Kirby_zero_handle_straightening}
    There exists a constant $\epsilon>0$ such that the following holds. If $f:B^n\to \mathbb{R}^n$ is an embedding such that $\|f-\id\|_{C^0}<\epsilon$, then there exists an embedding $g:B^n\to \mathbb{R}^n$ such that 
    \begin{enumerate}
        \item $g = \id$ on $B^n([2/3,1])$.
        \item $g = f$ on $B^n(1/3)$.
    \end{enumerate}
\end{Proposition}

\begin{proof}
We will focus on explaining the main idea of the argument, and will omit some details in point-set topology.  
Let $T^n$ be the $n$--dimensional torus, and let 
    \[
    \iota:B^n\hookrightarrow T^n
    \]
    be a smooth embedding.
    By the Hirsch--Smale theorem, there exists an immersion 
    \[
    i: T^n\setminus \iota(\mathring{B}(1/4)) \to B^n.
    \]
By a straightforward argument in point-set topology, when $\|f-\id\|_{C^0}$ is sufficiently small, there exists an embedding
\[
    f_1: T^n \setminus \iota(\mathring{B}(1/2))\to T^n \setminus \iota(\mathring{B}(1/4)),
    \]
    such that the following diagram commutes:
\[
\begin{tikzcd}
    T^n \setminus \iota(\mathring{B}(1/2)) \arrow[r, "f_1"] \arrow[d, "i"] & T^n \setminus \iota(\mathring{B}(1/4)) \arrow[d, "i"]\\
    B^n \arrow[r, "f"] & \mathbb{R}^n.
\end{tikzcd}
\]

Consider the restriction of $f_1$ to $T^n \setminus \iota(\mathring{B}(3/4))$. The image of the boundary is then a bi-collared embedded sphere in $T^n$, so it bounds a ball by the Schoenflies theorem. As a result, there exists a homeomorphism
\[
f_2:T^n\to T^n
\]
such that $f_1= f_2$ on $T^n \setminus \iota(\mathring{B}(3/4))$.

Let $\tilde{f}_2:\mathbb{R}^n\to \mathbb{R}^n$ be a lifting of $f_2$ to the universal cover of $T^n$.  Let $\eta:\mathbb{R}^n\to \mathring{B}^n(2/3)$ be a homeomorphism that is radial near infinity.
Then 
\[
f_3 = \eta\circ \tilde{f}_2\circ \eta^{-1}: \mathring{B}^n(2/3) \to \mathring{B}^n(2/3)
\]
is a homeomorphism. Moreover, because $|\tilde{f}_2(x)-x|$ is bounded for $x\in \mathbb{R}^n$, the map $f_3$ extends continuously to the boundary of $B^n(2/3)$ by the identity.  Let $g$ be the extension of $f_3$ to $B^n$ by the identity, then $g$ is a homeomorphism of $B^n$ and $g = \id$ on $B^n([2/3,1])$. 

For the rest of the proof, we show that under a suitable choice of $i$, $\eta$, and the lifting of $f_2$, we have $g=f$ on $B^n(1/3)$. Let us first explain the general idea. We may choose $i$ so that there is a closed ball $D\subset T^n\setminus \iota(B(1))$ that is mapped diffeomorphically to $B^n(1/2)$. Let $D'\subset D$ be the preimage of $B^n(1/3)$ in $D$. If we assume $\|f-\id\|_{C^0}$ is sufficiently small so that $f(B^n(1/3))\subset B^n(1/2)$, then $f_1(D')\subset D$. In this case, all information of $f|_{B^n(1/3)}$ is captured by $f_1|_{D'}$ and subsequently transferred to $g$. Thus, for a carefully chosen $\eta$, the map $g$ will be equal to $f$ on $B^n(1/3)$. Here are the details of the proof: 

We have a commutative diagram
\[
\begin{tikzcd}
    D' \arrow[r, "f_1"] \arrow[d, "i"', "\cong"] & D \arrow[d, "i"', "\cong"]\\
    B^n(1/3) \arrow[r, "f"] & B^n(1/2).
\end{tikzcd}
\]
Let $\widetilde{D}$ be a connected component of the lifting of $D$ in $\mathbb{R}^n$, and let $\widetilde{D}'\subset \widetilde{D}$ be the corresponding component of the lifting of $D'$. We may choose $\tilde{f}_2$ so that $\tilde{f}_2(\widetilde{D}')\subset \widetilde{D}$. Choose $\eta$ so that the composition 
\begin{equation}
\label{eqn_Kirby_zero_handle_straighten_def_eta}
B^n(1/2) \xrightarrow[\cong]{(i|_D)^{-1}} D \cong \widetilde{D} \xrightarrow{\eta} \mathring{B}^n(2/3) 
\end{equation}
is the identity embedding. Such $\eta$ exists because every smooth embedding of $B^n$ in $\mathbb{R}^n$ is smoothly isotopic to the standard embedding. We have a commutative diagram
\[
\begin{tikzcd}
    B^n(1/3) \arrow[r, "(i|_D)^{-1}", "\cong"'] \arrow[d, "f"] & D' \arrow[d, "f_1=f_2"] & \widetilde{D}' \arrow[l, "\cong "'] \arrow[d, "\tilde{f}_2"] \arrow[r, "\eta"] & \mathring{B}^n(2/3) \arrow[d, "g"]\\
    B^n(1/2) \arrow[r, "(i|_D)^{-1}", "\cong"'] & D  & \widetilde{D} \arrow[l, "\cong "'] \arrow[r, "\eta"] & \mathring{B}^n(2/3).
\end{tikzcd}
\]
Therefore, $g=f$ on $B^n(1/3)$.
\end{proof}

\begin{remark}
    If $f$ is a diffeomorphism or symplectomorphism, the map $g$ produced by the proof of Proposition \ref{prop_Kirby_zero_handle_straightening} may not be a diffeomorphism or symplectomorphism. Therefore, the argument of this paper cannot prove an analogous result to Theorems \ref{thm_main_bilip} that replaces ``bi-Lipschitz structures'' with smooth or symplectic structures. In fact, by Remark \ref{rmk_n=4} and the smoothing theory for PL homeomorphisms, every topological manifold with dimension $n \neq 4$ admits a chart whose transition functions are $C^0$--limits of diffeomorphisms, but there exist topological $n$--manifolds that do not support smooth structures. 
\end{remark}

\subsubsection{The idea of the proof of extension after restriction; further inputs for the proof}

We can now explain the idea behind Sullivan's proof of Proposition \ref{thm_match_handle_in_C}.   If $f$ is a $\mathcal{C}$ embedding, then in general, the map $g$ given by the above proof of Proposition  \ref{prop_Kirby_zero_handle_straightening} may not be a $\mathcal{C}$ embedding. However, one may modify the proof of Proposition  \ref{prop_Kirby_zero_handle_straightening} so that the construction of $g$ preserves $\mathcal{C}$--smoothness. This is one of the central arguments of \cite{sullivan1979hyperbolic}: If we replace $T^n$ with a closed connected oriented hyperbolic manifold $Q^n$ such that the tangent bundle of $Q^n\setminus \pt$ is trivial, then the Hirsch--Smale theorem still applies to give the immersion $i$. In this case, the universal cover of $Q^n$ is isometric to the hyperbolic space form $\mathbb{H}^n$, which is conformally isomorphic to $\mathring{B}^n$.  If the map $\eta$ in the proof of Proposition  \ref{prop_Kirby_zero_handle_straightening} is a quasi-conformal diffeomorphism, then the construction of $g$ preserves $\mathcal{C}$--smoothness. 

We collect the following antecedent facts here needed to make the sketch in the previous paragraph rigorous, 
and we make some remarks about their references.  
The first proposition establishes the existence of the desired hyperbolic manifold.

\begin{Proposition}[\cite{sullivan1979hyperbolic}]
\label{prop_almost_para_hyperbolic}
    For each $m\ge 2$, there exists a closed, connected, $m$--dimensional hyperbolic manifold $Q^m$ such that the tangent bundle of $Q^m\setminus \pt$ is trivial.
\end{Proposition}

\begin{remark}
    Proposition \ref{prop_almost_para_hyperbolic} is a deep result whose proof involves mod p algebraic geometry. However, if we only consider $4$--manifolds, then we only need Proposition \ref{prop_almost_para_hyperbolic} for the cases $m=2,3,4$. In these cases, the desired manifold $Q^m$ can be constructed directly:  If $m=2$ or $3$, every closed, connected, oriented hyperbolic $m$--manifold satisfies the desired property. If $m=4$, then a closed, connected, oriented hyperbolic $4$--manifold satisfies the desired property if and only if it is spin, and one such example is the Davis hyperbolic $4$--manifold (see \cite{ratcliffe2001davis}).
\end{remark}

The next proposition collects what we need to know concerning the existence of the aforementioned conformal isometry between the universal cover of $Q^n$ and the open disc.  

To be precise, for $m\ge 2$, let $Q^m$ be a fixed hyperbolic manifold given by Proposition \ref{prop_almost_para_hyperbolic}. For $m=1$, let $Q^m=S^1$. Let $\widetilde{Q}^m$ denote the universal cover of $Q^m$. For each $m\ge 1$, fix a smooth embedding $\iota_0:B^m(1/2)\hookrightarrow Q^m$, and let $\tilde{\iota}_0$ be a fixed lifting of $\iota_0$ to $\widetilde{Q}^m$. If $n>m$, we view $\mathbb{R}^m$ as a subspace of $\mathbb{R}^n$ by the standard embedding, and view $B^m$ as a subspace of $B^n$ accordingly.

\begin{Proposition}[\cite{sullivan1979hyperbolic}, see also {\cite[Section 2]{tukia1981lipschitz}}]
    \label{prop_coordinate_change_Bk_times_Qm}
    Suppose $k\ge 0$, $m\ge 1$, $n=m+k$. There exists a quasi-conformal diffeomorphism 
    \[
    \mathscr{U}: \widetilde{Q}^m \times B^k \to B^n\setminus \partial B^m,
    \]
    such that the following statements hold. 
    \begin{enumerate}
        \item Suppose $F:Q^m\times B^k\times I\to Q^m\times B^k$ is a homotopy relative to a neighborhood of $Q^m\times \partial B^k$ with $F(-,0) = \id$, $F(-,1) = f$, such that $f$ is a homeomorphism. Let $\widetilde{F}: \widetilde{Q}^m\times B^k\times I\to \widetilde{Q}^m\times B^k$ be the lifting\footnote{The existence and uniqueness of $\widetilde{F}$ follow from the homotopy lifting property of covering spaces.} of $F$ such that $\widetilde{F}(-,0)=\id$.
         Let $\tilde{f} = \widetilde{F}(-,1)$, then $\tilde{f}$ is a homeomorphism and a lifting of $f$. 
        Let $\hat f$ be the extension of 
        \[
        \mathscr{U}\circ \tilde{f} \circ \mathscr{U}^{-1}(x) :B^n\setminus \partial B^m\to B^n\setminus \partial B^m
        \]
        to $\mathbb{R}^n$ by the identity. The $\hat f$ is always a homeomorphism of $\mathbb{R}^n$.
\item $\hat f$ depends continuously on $f$ with respect to the $C^0$--norm.
\item If $f$ is a $\mathcal{C}$ homeomorphism, then $\hat f$ is a $\mathcal{C}$ homeomorphism.
\item There is a bi-Lipschitz homeomorphism 
\[
\eta: B^n \cup N_{1/4}(\partial B^m) \to B^m(1/2)\times B^k,
\]
where $N_{1/4}(\partial B^m)$ denotes the closed $1/4$--neighborhood of $\partial B^m$ in $\mathbb{R}^n$, 
such that the composition 
\[
B^m(1/3)\times B^k \xrightarrow{\tilde{\iota}_0\times \id } \widetilde{Q}^m\times B^k \xrightarrow{\mathscr{U}}B^n\setminus \partial B^m \xrightarrow{\eta} B^m(1/2)\times B^k
\]
is the identity embedding. 
\end{enumerate}
\end{Proposition}

\begin{remark}
    We can take the map $\mathscr{U}$ to be the inverse of the map $u$ in \cite[Section 2, page 308]{tukia1981lipschitz}. The fact that $\mathscr{U}$ is quasi-conformal follows from \cite[Lemma 2.11]{tukia1981lipschitz}. Statements 2, 4 follow directly from the definition of $u$. Statement 1 follows from \cite[Lemma 2.14]{tukia1981lipschitz}. Statement $3$ can be deduced as follows: (1) if $\mathcal{C}$ denotes locally bi-Lipschitz maps, the desired result is proved by  \cite[Lemma 2.12]{tukia1981lipschitz}; (2) if $\mathcal{C}$ denotes locally quasi-conformal maps, then since $\mathscr{U}$ is (globally) quasi-conformal, we have that $\hat f$ is a homeomorphism that is (globally) quasi-conformal on $\mathbb{R}^n\setminus \partial B^n$. The desired result then follows from the removability of singularity theorem for quasi-conformal homeomorphisms \cite[Theorem 35.1]{vaisala2006lectures}.
\end{remark}

Now we review the proof of Theorem \ref{thm_match_handle_in_C}.

\begin{repTheorem}{thm_match_handle_in_C}[\cite{sullivan1979hyperbolic}, see also {\cite[Theorem 3.3]{tukia1981lipschitz}}]
Suppose $m,k$ are non-negative integers such that $m+k=n$.
Let $\epsilon_0\in(0,1/2)$ be a constant.    
Then there exists an open neighborhood $\mathcal{P}$ of 
\[
\id \in \Emb(B^m\times B^k, B^m\times (B^k\setminus B^k(1-\epsilon_0));\mathbb{R}^n),
\]
a neighborhood $N$ of $\partial(B^m(1/2)\times B^k)$ in $B^m(1/2)\times B^k$, and a map
    \[
\psi:\mathcal{P} \to \Emb(B^m(1/2)\times B^k, N; \mathbb{R}^n),
   \]
    such that 
   \begin{enumerate}
       \item $\psi(\id) = \id$.
       \item $f = \psi(f)$ on $B^m(1/4)\times B^k$ for all $f\in \mathcal{P}$.
        \item If $f\in \mathcal{P}$ is a $\mathcal{C}$ embedding, then $\psi(f)$ is a  $\mathcal{C}$ embedding. 
    \end{enumerate}
\end{repTheorem}

\begin{proof}
    If $m=0$, we can take $\psi(f)$ to be the conjugation of $f$ by a fixed radial homeomorphism of $B^n$.
    In the following, assume $m>0$.    
    For $f\in \Emb(B^m\times B^k, B^m\times \partial B^k;\mathbb{R}^n)$ sufficiently close to the identity, we construct the desired $\psi(f)$ in the following steps.\\

    {\bf Step 1.} Recall that $\iota_0$ is a fixed smooth embedding of $B^m(1/2)$ in $Q^m$. Let $\iota_1:B^m\to Q^m$ be another fixed smooth embedding whose image is disjoint from the image of $\iota_0$. For $r\in (0,1]$, let $D(r)\subset Q^m$ be the image of $B^m(r)$ under $\iota_1$. To simplify notation, let $Q^m(r) = Q^m\setminus \mathring{D}(r)$ for $r\in(0,1]$.

    Since $Q^m\setminus \pt$ has a trivial tangent bundle, by the Hirsch--Smale theorem, there exists a smooth immersion 
    \[
    i: Q^m(1/2) \to B^m
    \]
    such that $i \circ \iota_0: B^m(1/4)\to B^m$ is the identity embedding. Fix the immersion $i$  from now.

Let
    \[
\delta_0 = \inf \{\dist_{Q^m}(x,y)\mid x\neq y \in Q^m(1/2), \textrm{ and }i(x) = i(y)\}.
\]
Since $i$ is a local embedding and $Q^m(1/2)$ is compact, we have $\delta_0>0$.  

By a straightforward argument which is now standard (see, for example, \cite{edwards1971deformations}, the definition of $h_1$ on page 68), if 
\[
f\in\Emb(B^m\times B^k, B^m\times (B^k\setminus B^k(1-\epsilon_0));\mathbb{R}^n)
\]
is sufficiently close to $\id$ in the $C^0$--norm, there exists a unique 
map $\psi_1(f): Q^m(3/4) \to Q^m(1/2)$ such that $\dist_{Q^m}(\psi_1(f)(x),x)<\delta_0/2$ for all $x$ and the following diagram commutes:
\[
\begin{tikzcd}
    Q^m(3/4)\times B^k \arrow[r, "\psi_1(f)"] \arrow[d, "i\times \id"] & Q^m(1/2)\times B^k \arrow[d, "i\times\id"] \\
    B^m\times B^k \arrow[r, "f"] & B^m\times B^k.
\end{tikzcd}
\]
Moreover, $\psi_1(f)$ is a continuous embedding and depends on $f$ continuously in the $C^0$--topology. If $f$ is a $\mathcal{C}$ embedding, then $\psi_1(f)$ is a $\mathcal{C}$ embedding. We also have $\psi_1(\id) = \id$. 

{\bf Step 2.} 
Note that $\psi_1(f) = \id$ on $Q^m(3/4)\times B^k\cap Q^m\times (B^k\setminus B^k(1-\epsilon_0))$.
Let
\begin{multline*}
    \psi_2(f):Q^m(3/4)\times B^k\cup Q^m\times(B^k\setminus B^k(1-\epsilon_0)) \\
    \to Q^m(1/2)\times B^k\cup Q^m\times (B^k\setminus B^k(1-\epsilon_0))
\end{multline*}
be the extension of $\psi_2(f)$ by the identity. Then $\psi_2(f)$ is an embedding if $f$ is sufficiently close to $\id$, and $\psi_2(f)$ is a $\mathcal{C}$ embedding if $f$ is a $\mathcal{C}$ embedding. 

{\bf Step 3.} Let  $A^n$ be the closure of $(D(1)\times B^k)\setminus (D(3/4)\times B^k(1-\epsilon_0))$, let $D^n = D(1)\times B^k$. Then $A^n\subset D^n$, and there exists a bi-Lipschitz homeomorphism that takes $D^n$ to $B^n$ such that the image of $A^n$ is $B^n([3/4,1])$. Therefore, by the canonical Schoenflies theorem (Theorem \ref{thm_canonical_schoenflies}), when $f$ is sufficiently close to $\id$, there is an embedding 
\[
\varphi(f): D(1)\times B^k \to Q^m\times B^k
\]
depending continuously on $f$, such that the extension of $\varphi(f)$ to $Q^m\times B^k$ by $\psi_2(f)$, which we denote by $\psi_3(f)$, is a homeomorphism depending continuously on $f$, and $\psi_3(f)$ is a $\mathcal{C}$ homeomorphism if $f$ is a $\mathcal{C}$ homeomorphism. Moreover, $\psi_3(f)$ equals the identity on a neighborhood of $Q^m\times \partial B^k$, where the neighborhood only depends on $\|f-\id\|_{C^0}$ and is independent of $f$.

{\bf Step 4.} Let $\mathscr{U}$, $\eta$ be as in Proposition \ref{prop_coordinate_change_Bk_times_Qm}. If $f$ is $C^0$--close to $\id$, then $\psi_3(f)$ is $C^0$--close to $\id$, so $\psi_3(f)$ can be homotoped to $\id$ by local geodesics. Let 
\[
\tilde{\psi}_3(f):\widetilde{Q}^m\times B^k\to \widetilde{Q}^m\times B^k
\]
be the lifting of $\psi(f)$ constructed by lifting the geodesic homotopy from $\id$ to $\psi_3(f)$ (see Proposition \ref{prop_coordinate_change_Bk_times_Qm} (1)), and let $\hat{\psi}_3(f):\mathbb{R}^n\to \mathbb{R}^n$ be the extension of $\mathscr{U}\circ \tilde{\psi}_3(f) \circ \mathscr{U}^{-1}$ by the identity. Recall that $\hat{\psi}_3(f)$ is a homeomorphism that equals the identity on the complement of $B^n$.  Define $\psi(f)$ to be the composition 
\[
B^m(1/2)\times B^k \xrightarrow{\eta^{-1}}B^n \cup N_{1/4}(\partial B^m)\xrightarrow{\hat{\psi}_3(f)}B^n \cup N_{1/4}(\partial B^m)\xrightarrow{\eta}B^m(1/2)\times B^k.
\]
Then there is an open neighborhood $N$ of $\partial (B^m(1/2)\times B^k)$ in $B^m(1/2)\times B^k$, which is independent of $f$, such that $\psi(f)$ equals the identity on $N$. 
Suppose $f$ is sufficiently close to $\id$ such that $f(B^m(1/4)\times B^k)\subset B^m(1/3)\times B^k$.
Then we have the following commutative diagram
\[
\begin{tikzcd}
    B^m(1/4)\times B^k \arrow[r, "\tilde{\iota}_0\times \id"] \arrow[d, "f"] & \widetilde{Q}^m\times B^k \arrow[r, "\mathscr{U}"] \arrow[d, "\tilde{\psi}_3(f)"] & B^n\setminus \partial B^m \arrow[r, "\eta"] \arrow[d, "\hat{\psi}_3(f)"] & B^m(1/2)\times B^k \arrow[d, "\psi(f)"]\\
    B^m(1/3)\times B^k\arrow[r, "\tilde{\iota}_0\times \id"] & \widetilde{Q}^m\times B^k \arrow[r, "\mathscr{U}"] & B^n\setminus \partial B^m \arrow[r, "\eta"] & B^m(1/2)\times B^k.
\end{tikzcd}
\]
By Statement (4) of Proposition \ref{prop_coordinate_change_Bk_times_Qm}, the composition of the horizontal maps equals the identity embedding, so we have $\psi(f) = f$ on $B^m(1/4)\times B^k$. It is also clear from the construction that $\psi(\id)=\id$. 
\end{proof}

\bibliographystyle{amsalpha}
\bibliography{references}

\end{document}